\documentclass[12pt]{amsart}

\usepackage{CJK}

\newtheorem{theorem}{Theorem}[section]
\newtheorem{lemma}[theorem]{Lemma}
\newtheorem{proposition}[theorem]{Proposition}
\newtheorem{corollary}[theorem]{Corollary}

 \theoremstyle{definition}
\newtheorem{definition}[theorem]{Definition}

\theoremstyle{remark}
\newtheorem{remark}[theorem]{Remark}

\numberwithin{equation}{section}

\begin{document}
\begin{CJK*}{GBK}{song}

\title[Prescribed curvature problem on complete noncompact manifolds]
{Local $C^0$-estimate and existence theorems for some prescribed curvature problems
on complete noncompact Riemannian manifolds}
%{Local $C^0$-estimate and existence of solutions to some prescribed curvature equations on complete noncompact Riemannian manifolds}
%{Some prescribed curvature equations on complete noncompact Riemannian manifolds}
%{The prescribed curvature problem on complete noncompact Riemannian manifolds}
%{On the prescribed curvature problem of complete noncompact conformal metrics}
%{The prescribed curvature problem for complete noncompact conformal metrics}
%{On the complete noncompact conformal metrics with prescribed curvature functions}
%{On the conformal deformation with prescribed curvature functions on complete  noncompact Riemannian manifolds}
%{On the conformal metrics with prescribed curvature functions on complete  noncompact Riemannian manifolds}
%{On the conformal metrics with prescribed curvature functions  II: complete and noncompact case}
%{Conformal deformation to prescribed curvature functions on complete noncompact Riemannian manifolds }
%{The conformal metrics with prescribed curvature functions  II: complete and noncompact case}
%{Conformal metrics with prescribed curvature functions on complete noncompact Riemannian manifolds}
%{Prescribing curvature problem on the conformal classes of complete metrics}
%{Prescribing curvature as a function on complete noncompact Riemannian manifolds}
%{Complete conformal metrics with prescribed curvature as functions on complete noncompact Riemannian manifolds}
 
 \author{Rirong Yuan}
\address{School of Mathematics, South China University of Technology, Guangzhou 510641, Guangdong, China}
\email{yuanrr@scut.edu.cn}
%\email{yuanrr@scut.edu.cn; rirongyuan@stu.xmu.edu.cn}

\thanks{Research supported in part by NSFC grant 11801587.
 % The author was supported by %the National Natural Science Foundation of China, Grant No.
%  NSFC Grant 11801587.
}
\date{}

\begin{abstract}
%In this paper we study a problem of finding complete conformal metrics with prescribed curvature functions  on 
%complete noncompact Riemannian manifolds.  
In this article we study a class of prescribed curvature problems on complete noncompact Riemannian manifolds.
To be precise, we derive local $C^0$-estimate under an asymptotic condition which is in effect optimal, and prove the existence of complete conformal metrics with prescribed curvature functions.
A key ingredient of our strategy is  Aviles-McOwen's result or its fully nonlinear version
on the existence of complete conformal metrics with prescribed curvature functions on manifolds with boundary.
% The asymptotic behavior of complete conformal metrics near boundary is also formulated.

%A key ingredient is Loewner-Nirenberg and Aviles-McOwen's result 
%on the existence of complete conformal metric of negative scalar curvature on manifolds with boundary.
% The asymptotic behavior of complete conformal metrics near boundary is also formulated.

%{\em Mathematical Subject Classification (2010):}
%{\em Mathematics Subject Classification (2010):} 53C21, 58J05, 53A30. %58J05, 35J60, 35J15, 35B45.

%{\em Keywords:} Complete noncompact Riemannian manifold, modified Schouten tensor, conformal change, fully nonlinear equations

\end{abstract}

\maketitle

 \section{Introduction}

%which function could arise as a curvature function of a metric which is complete and conformally equivalent to the given metric

An important problem in conformal geometry is to determine which function could arise as a curvature function of a metric that is complete and conformally equivalent to a given metric.
%the set of curvature functions for metrics preserving conformal class that a manifold may possess. 
A special case, known as Yamabe problem, conjectures that any closed Riemannian manifold %(compact and without boundary)
 of  dimension $\geq3$ can be conformally deformed to achieve constant scalar curvature. The Yamabe problem on closed  manifolds is valid according to the work of Trudinger \cite{Trudinger1968}, Aubin \cite{Aubin1976} and Schoen \cite{Schoen1984}, while for general complete noncompact Riemannian manifolds the corresponding Yamabe problem is not always true as shown by Jin \cite{Jin1988} who constructed some counterexamples.
This reveals that there are great differences between the conformal change of Riemannian metrics on closed manifolds and those on complete noncompact manifolds.
 Therefore,  
  it is reasonable to investigate which complete noncompact Riemannian manifold could admit a complete conformal metric 
  with the scalar curvature being a constant or a more general function.
  The problem in this and related topic has attracted enormous  interest,  %among others
 see  for instance 
  \cite{Aviles1985McOwen,Aviles1988McOwen2,Cheng1992Ni,Jin1993,Ni-82Indiana} and references therein.
  
%It should be stressed that in Theorem ?? there is no assumption on ?? 

This paper is devoted to study a fully nonlinear version and then extend it to higher order curvatures.  
%more general curvature tensors of second order.

Without specific clarification, we assume throughout this article that $(M,g)$ is a  complete noncompact Riemannian manifold of dimension $n\geq3$ with the Levi-Civita connection 
$\nabla$. For a Riemannian metric $g$, we denote the sectional, Ricci and scalar curvature tensors by  $Sec_g$,
$Ric_g$ and $R_g$, respectively. 
%In addition, we use the notation $v\preceq w$ to denote that
%there is a uniformly positive constant $\delta$ such that $v\leq \delta w$ in $M$. The notation $\succeq$ is obvious. 

Partial results are stated as follows.

Let  $\sigma_k$ be the $k$-th elementary symmetric function,   $\Gamma_k$
 be the $k$-th G{\aa}rding's cone 
 $$\Gamma_k=\left\{\lambda\in \mathbb{R}^n: \sigma_j(\lambda)>0, \mbox{  } \forall 1\leq j\leq k \right\}.$$

 \begin{theorem}
\label{thm2-ricci}
%Suppose %$(M,g)$ %be a complete noncompact Riemannian manifold 
On a complete noncompact Riemannian manifold $(M,g)$  
satisfying 
\begin{equation}  \begin{aligned}
\,& \sigma_k(\lambda(-g^{-1}Ric_g))\geq \delta>0,  \,& \lambda(-g^{-1}Ric_g)\in\Gamma_k  \nonumber
\end{aligned}  \end{equation}
for some $2\leq k\leq n$ and for some constant $\delta>0$,
 there exists at least one smooth complete conformal metric $\tilde{g}=e^{2u}g$ with  
 $\lambda(-\tilde{g}^{-1}Ric_{\tilde{g}}) \in\Gamma_k$ and 
\begin{equation}
 \begin{aligned}
 \sigma_k(\lambda(-\tilde{g}^{-1}Ric_{\tilde{g}}))=\psi \mbox{ in } M, \nonumber
 \end{aligned} 
 \end{equation}
provided that $\psi$ is a positive smooth function having a uniform bound 
$0<\psi<\Lambda,$ where $\Lambda$ is a positive constant.
%\begin{equation}\label{psi-1} \begin{aligned}
%0<\psi<\Lambda, \mbox{ where $\Lambda$ is a positive constant}. 
% \end{aligned}  \end{equation}
\end{theorem}
%Combining with a theorem of Lohkamp \cite{Lohkamp1994}, we obtain
%\begin{corollary}
%Any noncompact manifold of dimension $n\geq3$ admits a complete noncompact Riemannian metric $g$ with 
%negative Ricci curvature
%$Ric_g<0$ and $\det(-{g}^{-1}Ric_{{g}})=1.$
%\end{corollary}
By studying $\sigma_k$ curvature equation for conformal deformation of  Einstein tensor,
 $G_g=Ric_g-\frac{R_g}{2}g,$ we also deduce
\begin{theorem}
\label{thm3}
Let $(M,g)$ be a complete noncompact Riemannian manifold with %strictly negative sectional curvature
  $Sec_g\leq -\delta<0$ for some constant $\delta>0$.
Then for each $2\leq k\leq n-1$ and for any smooth positive function $\psi$ %satisfying \eqref{psi-1},
 with $0<\psi<\Lambda$ for some constant $\Lambda>0$,
 $M$ admits a smooth complete conformal metric $\tilde{g}=e^{2u}g$ with  
  %$R_{\tilde{g}}<0$ and $(n^2-5n+8)R_{\tilde{g}}^2-4|Ric_{\tilde{g}}|_{\tilde{g}}^2=\psi.$
   \begin{equation}   \begin{aligned}
\,&  \sigma_k(\lambda(\tilde{g}^{-1}G_{\tilde{g}}))=\psi, \,& \lambda(\tilde{g}^{-1}G_{\tilde{g}})\in \Gamma_k\,& \mbox{ in } M.  \nonumber
 \end{aligned}  \end{equation}

\end{theorem}

In fact we are primarily concerned with a more general problem of finding 
 a complete conformal metric %$\tilde{g}=e^{2{u}}g$ 
with a prescribed curvature function
% with respect to $\tilde{g}$,
% In this paper we study the existence of complete conformal metric $\tilde{g}=e^{2\tilde{u}}g$ satisfying
   \begin{equation}
 \label{main-equ1}
 \begin{aligned}
\,& f(\lambda(\tilde{g}^{-1}A_{\tilde{g}}^{\tau,\alpha}))=\psi \mbox{  } \mbox{ in } M,  \mbox{  } 
   0<\psi\in C^\infty(M),  \\
% \,&\lim_{r(x)\rightarrow+\infty}  {u}(x)=+\infty,
 \tilde{g} =  e^{2{u}}\,&g\mbox{ is complete and \textit{admissible} } (i.e. \mbox{  } \lambda({\tilde{g}}^{-1}A_{\tilde{g}}^{\tau,\alpha})\in \Gamma),
%  \mbox{ (i.e. $\tilde{g}$ is \textit{admissible})},
 \end{aligned}
 \end{equation}
 which is generated by a smooth symmetric function $f$ 
of eigenvalues of %$(\pm)$ modified Schouten tensor 
$A_{\tilde{g}}^{\tau,\alpha}$ with respect to $\tilde{g}$, where
 \begin{equation}
 \begin{aligned}
 \,& A_{\tilde{g}}^{\tau,\alpha}=\frac{\alpha}{n-2} (Ric_{\tilde g}-\frac{\tau  R_{\tilde g} }{2(n-1)} \tilde{g}), \,& \alpha=\pm1,  \mbox{  }\tau \in \mathbb{R}. \nonumber
 \end{aligned}
 \end{equation}
 %where $\lambda(\tilde{g}^{-1}A_{\tilde{g}}^{\tau,\alpha})$ denote the eigenvalues of 
% $A_{\tilde{g}}^{\tau,\alpha}$  with respect to $\tilde{g}$.
The equation is a fully nonlinear equation of similar type to Hessian equations going back to the work of Caffarelli-Nirenberg-Spruck \cite{CNS3} which studied Dirichlet problem %for $f(\lambda(D^2u))=\psi$ 
on smooth bounded domains of Euclidean spaces.
 In addition,  %as in \cite{CNS3},
  $f$ % is of $n$ real variables and
  is defined in an open symmetric and convex cone 
$\Gamma\subset\mathbb{R}^n$
with vertex at origin, $\partial \Gamma\neq \emptyset$, $\Gamma_n\subseteq\Gamma\subset  \Gamma_1$.
Moreover, %we assume that %$f$ satisfies 
the following three conditions are required:
\begin{equation}
\label{elliptic}
 f_{i}(\lambda):= %f_{\lambda_i}(\lambda)=
 \frac{\partial f}{\partial \lambda_{i}}(\lambda)> 0  \mbox{ in } \Gamma,\  1\leq i\leq n,
\end{equation}
\begin{equation}
\label{concave}
 f \mbox{ is  concave in } \Gamma,
\end{equation}
%and $f$ is a homogeneous  function of degree one, i.e.,
\begin{equation}
\label{homogeneous-1}
\begin{aligned}
f(t\lambda)=t f(\lambda) % \mbox{  } f(\lambda)>0,
 \mbox{ for any } \lambda\in\Gamma, \mbox{  } t>0.
 \end{aligned}
 \end{equation}
 Without loss of generality, we assume throughout the article that
\begin{equation}
\label{14}
\begin{aligned}
f|_{\partial \Gamma}\equiv0, \mbox{  } f|_{\Gamma}>0 \mbox{ and } f(\vec{\bf 1})=1, \nonumber
 \end{aligned}
 \end{equation} 
where %and hereafter  
$\vec{\bf 1}=  (1,\cdots, 1)\in \mathbb{R}^n$, whenever $f$ is homogeneous of degree one.
%obeying \eqref{homogeneous-1}, %i.e. $f$ obeys \eqref{homogeneous-1}. 
This implies
\begin{equation}\label{key1-main} \begin{aligned}
%\,& %\sum_{i=1}^n \lambda_i \geq \frac{n}{\gamma}f(\lambda)+n-\frac{n}{\gamma},
\,& \sum_{i=1}^n \lambda_i \geq n f(\lambda),  \,& \lambda\in\Gamma.
 \end{aligned} \end{equation}
% which follows from \eqref{homogeneous-1} and \eqref{concave}.

Typical examples satisfying \eqref{elliptic}, \eqref{concave}, \eqref{homogeneous-1} are given by 
$f=\left({C_n^l\sigma_k}/{(C_n^k\sigma_l})\right)^{{1}/{(k-l)}},$ 
$0\leq l<k\leq n, \mbox{ } \Gamma=\Gamma_k,$ where $\sigma_0=1$ and $C_n^k=n!/((n-k)!k!)$.

%The study of fully nonlinear equations of the type goes back to the work of Caffarelli-Nirenberg-Spruck \cite{CNS3} which studied Dirichlet problem for $f(\lambda(D^2u))=\psi$ on smooth bounded domains of Euclidean spaces.
The $k$-Yamabe problem for conformal deformation of  Schouten tensor $A_g=\frac{1}{n-2}(Ric_g-\frac{R_g}{2(n-1)} g)$ %with $\tau=\alpha=1$ 
was proposed by Viaclovsky \cite{Viaclovsky2000} and since then
  %the study of conformal deformation of Schouten tensor %and more general modified Schouten tensors
  % on compact Riemannian manifolds with or without boundary 
it  has drawn much attention as the interest from geometry, see for instance  
  \cite{ChangGurskyYang2002,Ge2006Wang,Guan2003Wang-CrelleJ,Gursky2007Viaclovsky,ABLi2003YYLi,ShengTrudingerWang2007}. %and references therein.

% \vspace{1mm}
 Before stating our more general results, we shall recall %the definition of
 %a notation $\kappa_\Gamma$ for $\Gamma$ and 
some result on partial uniform ellipticity. %introduced in \cite{yuan2020conformal}.
 \begin{definition}
 [\cite{yuan2020conformal}]
\label{yuan-kappa}
For the cone $\Gamma$, we denote % stated above, we define
\begin{equation}
\begin{aligned}
\kappa_{\Gamma}:=\max \left\{k: (-\alpha_1,\cdots,-\alpha_k,\alpha_{k+1},\cdots, \alpha_n)\in \Gamma, \mbox{ where } \alpha_j>0, \mbox{ } \forall 1\leq j\leq n \right\}. \nonumber
\end{aligned}
\end{equation}
\end{definition}
Obviously, $\kappa_\Gamma$ is an integer with $0\leq \kappa_\Gamma\leq n-1$.
%\begin{itemize}
%\item $\kappa_\Gamma$ is an integer with $0\leq \kappa_\Gamma\leq n-1$.
% \item $\kappa_\Gamma\geq 1$ if and only if $\Gamma\neq \Gamma_n$.
%\item $\kappa_\Gamma=0$ if and only if $\Gamma=\Gamma_n$.
%\end{itemize}
It was shown in \cite{yuan2020conformal} that
the constant $\kappa_\Gamma$ measures the partial uniform ellipticity of $f$:
\begin{lemma}
[\cite{yuan2020conformal}]
\label{yuan-k+1}
Let $f$ and $\Gamma$ be as above. %and $\kappa_\Gamma$ be as defined in Definition \ref{yuan-kappa}.
%Suppose $f$ satisfies \eqref{elliptic}, \eqref{concave} and \eqref{addistruc}.
Suppose $f$ satisfies \eqref{elliptic}, \eqref{concave}  and
\begin{equation}
\label{addistruc}
\begin{aligned}
\mbox{For each $\sigma<\sup_{\Gamma}f$ and } \lambda\in \Gamma, \mbox{   } \lim_{t\rightarrow +\infty}f(t\lambda)>\sigma.
\end{aligned}
\end{equation}
Then there is a universally positive constant $\vartheta_{\Gamma}$ depending only on %$n$ and 
$\Gamma$, such that
    for each $ \lambda\in \Gamma$ with order $\lambda_1 \leq \cdots \leq\lambda_n$,
   \begin{equation}
   \label{k+1elliptic}
   \begin{aligned}
   f_{{i}}(\lambda) \geq   \vartheta_{\Gamma} \sum_{j=1}^{n}f_j(\lambda) \mbox{ for }  1\leq i\leq \kappa_\Gamma+1. %\nonumber
    %\mbox{  } \cdots, \mbox{  } f_{{{1+\kappa_{\Gamma}}}}(\lambda) \geq  \vartheta_{\Gamma} \sum_{j=1}^{n}f_j(\lambda).
   \end{aligned}
   \end{equation}
%  In fact,  $f_{{1}}(\lambda) \geq   \frac{1}{n} \sum_{j=1}^{n}f_j(\lambda)$ and $0<\vartheta_{\Gamma}\leq\frac{1}{n}$. %and the specific choice for $\vartheta_{\Gamma}$ will be given in \eqref{theta1}.
Moreover, the statement \eqref{k+1elliptic} is sharp and cannot be further improved.

\end{lemma}

 Throughout this present paper, $\vartheta_{\Gamma}$ and $\kappa_\Gamma$ always stand
  for the constants  in Lemma \ref{yuan-k+1} and Definition \ref{yuan-kappa}, respectively.

%We now state our result as in the following.
We now impose an appropriate and delicate restriction to 
the parameters in $A_g^{\tau,\alpha}$: %$\tau$ should satisfy
\begin{equation}
\label{tau-alpha}
 \begin{aligned}
\,& \tau<1, \,&\mbox{ if } \alpha=-1; \\
\,& \tau>1+(n-2)(1-\kappa_\Gamma\vartheta_{\Gamma}), \,&\mbox{ if } \alpha=1,
 \end{aligned}
 \end{equation}
which further yields 
 \begin{equation}
 \label{positive1}
 \begin{aligned}
 \alpha(n\tau+2-2n)>0.
 \end{aligned}
 \end{equation}
 (It would be worthwhile to note that if $\alpha=1$ equation \eqref{main-equ1} fails to be elliptic for general $\tau<n-1$, $\tau\neq1$; surprisingly assumption \eqref{tau-alpha}
 allows $1+(n-2)(1-\kappa_\Gamma\vartheta_{\Gamma})<\tau< n-1$ in the case $\Gamma\neq\Gamma_n$).
Together with Lemma \ref{yuan-k+1} we can 
propose an effective way to understand the structure of linearized operator for equation \eqref{main-equ1} and then establish the local estimates for gradient and Hessian of solutions (Theorem \ref{thm1-local}).
 With such local estimates at hand, the remaining goal is to derive local zero order estimate for solutions to approximate Dirichlet problems. %\eqref{approximate-DP1}. 

In an attempt to derive local bound of solutions from below, we assume there exists an 
\textit{admissible} complete (noncompact) conformal metric $\underline{g}\in [g]$
satisfying
\begin{equation}
 \label{admissible-metric1}
 \begin{aligned}
 \lambda(\underline{g}^{-1}A_{\underline{g}}^{\tau,\alpha})\in \Gamma \mbox{ in } M,
 \end{aligned}
 \end{equation}
  where $[g]=\left\{e^{2v}g: v\in C^2(M)\right\}$ is the $C^2$-smooth complete conformal class of $g$,
and that such
 $\underline{g}$ further satisfies a key asymptotic property at infinity: There exists a compact subset $K_0$ of $M$ and a constant $\Lambda_0>0$ such that
 % on $M\setminus K_0$,
\begin{equation}
\label{key-assum1}
 \begin{aligned}
 \frac{f(\lambda(\underline{g}^{-1}A_{\underline{g}}^{\tau,\alpha}))}{\psi}\geq \Lambda_0  \mbox{ holds uniformly in $M\setminus K_0$.}
 \end{aligned}
 \end{equation}
 This asymptotic property near infinity can be viewed as a subsolution, which is only used to derive local bound for approximate solutions from below  (Theorem \ref{thm-c0-lower}).
 The assumption \eqref{key-assum1} is in effect sufficient and necessary for the existence result:
  %optimal and cannot be removed:
 \begin{itemize}
\item  The metric we expect to obtain in Theorem \ref{thm1} tautologically obeys \eqref{key-assum1}.
\item There is a counterexample on nonexistence of such a complete conformal metric, in which \eqref{key-assum1} is not satisfied. For some nonexistence results on scalar and negative Ricci curvatures we are referred for example to \cite{Jin1993,Ni-82Indiana,Sui2017JGA}. 
%\item It is new even on Euclidean spaces.
\end{itemize}

 %A key ingredient for deriving local bound of solutions from above is that we may apply 
 For the rest issue of deriving local bound of solutions from above,
  it requires to construct supersolutions. %To achieve it, 
  In this paper 
 we apply a theorem %(Theorem \ref{thm1-AM})
 due to Aviles-McOwen \cite{Aviles1988McOwen}, which extensively extends a famous result of  Loewner-Nirenberg \cite{Loewner1974Nirenberg},
%due to Loewner-Nirenberg \cite{Loewner1974Nirenberg} on a smooth bounded domain $\Omega\subset\mathbb{R}^n$ and 
%to Aviles-McOwen \cite{Aviles1988McOwen} for general compact Riemannian manifolds with smooth boundary,
 %(Theorem \ref{thm1-AM} below), 
  or its fully nonlinear version %on manifolds with boundary
   as shown in Theorem \ref{existence1-compact} below (partially proved in \cite{yuan2020conformal})
  %to construct supersolutions 
  to achieve the goal and thus provide 
   %a new insight and 
 a straightforward and simple approach to establish local bound of approximate solutions from above (Theorem \ref{thm-c0-upper}). 
 %Based on this observation we give two proofs.
% This observation is a contribution we achieve here.
It would be worthwhile to note that our strategy 
%method %based on the observation %for deriving upper bound does %not depend on 
does work without any further %additional %geometric (curvature) %and topological
geometric restrictions to underlying manifolds as well as to prescribed curvature function,
 %This is a contribution we achieve in this present paper,
 to be compared with those in \cite{Aviles1985McOwen,Jin1993} for scalar curvature equation and in \cite{FuShengYuan} for a generalization to $\sigma_k$ curvature equations for some $(0,2)$-type tensors.
% which is rather different from that was used in \cite{Aviles1985McOwen,Jin1993} for scalar curvature and in \cite{FuShengYuan} for a generalization to $\sigma_k$ curvature equations for some $(0,2)$-type tensors. 
 This is new even for Euclidean spaces, and it also works for prescribed scalar curvature equation (see Theorem \ref{thm-scalarcurvature}).
  %Also, the idea and strategy %developed here 
% should adapt to more general curvature equations in other contexts. %This will be investigated in future work.

%\vspace{1mm}
 As a result, we prove the main result. (Throughout the paper, we always assume that $f$ satisfies \eqref{elliptic}, \eqref{concave}, \eqref{homogeneous-1}, and  $\psi$ is the prescribed curvature function without specific clarification).
 
\begin{theorem}
%[Main result]
\label{thm1}
%Suppose $f$ satisfies \eqref{elliptic}, \eqref{concave}, \eqref{homogeneous-1}. 
For the  $\tau$ satisfying \eqref{tau-alpha}, we assume
 there is a complete conformal metric to satisfy \eqref{admissible-metric1} and \eqref{key-assum1}.
Then for such $\tau$,  
there exists at least one smooth function $u_\infty\in C^\infty(M)$ such that $g_\infty=e^{2u_\infty}g$ is a smooth  admissible complete
conformal metric on $M$ satisfying \eqref{main-equ1}.
\end{theorem}

This theorem asserts the existence of a smooth \textit{admissible} complete conformal metric solving the prescribed curvature equation \eqref{main-equ1}, provided that there is a $C^2$-smooth \textit{admissible} complete conformal metric $\underline{g}$ 
satisfying \eqref{key-assum1}.
%\begin{equation}\label{key-subsolution-g}\begin{aligned}
%f(\lambda(\underline{g}^{-1}A_{\underline{g}}^{\tau,\alpha}))\geq \delta \psi \mbox{ for some constant } \delta>0.
%\end{aligned}\end{equation} 
%A somewhat surprising fact to us is that 
%there are no further assumptions on underlying manifolds and prescribed curvature function.
That is, all of the geometric and analytic obstructions to the 
%existence of a smooth \textit{admissible} complete conformal metric solving 
solvability of %the prescribed curvature equation 
\eqref{main-equ1} are in fact embodied in the assumption of the asymptotic condition \eqref{key-assum1}.
% the existence of a $C^2$-smooth \textit{admissible} complete conformal metric $\underline{g}$ satisfying \eqref{key-subsolution-g}.
%with $f(\lambda(\underline{g}^{-1}A_{\underline{g}}^{\tau,\alpha}))\geq \delta \psi$ for some constant $\delta>0$.

Immediately we have Theorem \ref{thm2-ricci}, Corollaries \ref{coro1-4} and \ref{coro1-2}. Together with %the formula 
\eqref{GSW1},
we further derive Theorem \ref{thm3} and Corollaries \ref{coro1-1}, \ref{coro1-3}.

\begin{corollary}
\label{coro1-4}
Let $0<\psi\in C^\infty(M)$, and we assume
  in addition that $Ric_g \leq -\delta\psi$ for some constant $\delta>0$.
Then $M$ admits a smooth complete conformal metric $\tilde{g}=e^{2u}g$ 
 satisfying $f(\lambda(-\tilde{g}^{-1}Ric_{\tilde{g}}))=\psi$, $\lambda(-\tilde{g}^{-1}Ric_{\tilde{g}})\in\Gamma$.
  \end{corollary}

\begin{corollary}
\label{coro1-2}
%Assume \eqref{elliptic}, \eqref{concave}, \eqref{homogeneous-1} hold. 
Assume $\psi$ is a smooth positive function. %$0<\psi\in C^\infty(M)$.
%Let $\alpha=-1$, $\tau\leq 0$.
Suppose in addition that $Ric_g\leq 0$, $R_g\leq -\delta\psi$ for some constant $\delta>0$.
Then for such  $\psi$ and for $\alpha=-1$, $\tau< 0$,
%on $M$ with $0<\psi<\Lambda$ for some constant $\Lambda>0$,
there exists a smooth complete conformal metric %$\tilde{g}=e^{2u}g$
 satisfying \eqref{main-equ1}.
 \end{corollary}

\begin{corollary}
\label{coro1-1}
Let $\alpha=1$, $\tau> n-1$.
%Assume \eqref{elliptic}, \eqref{concave}, \eqref{homogeneous-1} hold.
Suppose $Sec_g \leq0$ and %negative scalar curvature
 $R_g\leq-\delta\psi$  for a smooth positive function $\psi$ for some constant $\delta>0$.
Then for such $\psi$, 
%on $M$ with $0<\psi<\Lambda$ for some constant $\Lambda>0$,
there is a smooth complete conformal metric $\tilde{g}=e^{2u}g$ solving \eqref{main-equ1}.
 \end{corollary}

% For $\tau=n-1$, which requires $\Gamma\neq\Gamma_n$ according to \eqref{tau-alpha},
% (also refer to \cite{yuan2020conformal} for a topological obstruction on certain $3$-manifolds with boundary),
 %we derive
\begin{corollary}
\label{coro1-3}
%Let $f$ satisfy \eqref{elliptic}, \eqref{concave}, \eqref{homogeneous-1}.
Let $0<\psi\in C^\infty(M)$.
Suppose $\Gamma\neq\Gamma_n$ and $Sec_g\leq -\delta\psi$ for some constant $\delta>0$.
Then for such $\psi$,
%on $M$ with $0<\psi<\Lambda$ for some constant $\Lambda>0$,
there is a smooth complete conformal metric $\tilde{g}=e^{2u}g$ with
$ \lambda(\tilde{g}^{-1}G_{\tilde{g}})\in\Gamma$ solving 
$f(\lambda(\tilde{g}^{-1}G_{\tilde{g}}))=\psi \mbox{ in } M.$
%\begin{equation}\label{equ1-Einstein} \begin{aligned}
% \lambda(\tilde{g}^{-1}G_{\tilde{g}})\in\Gamma \mbox{ and } f(\lambda(\tilde{g}^{-1}G_{\tilde{g}}))=\psi \mbox{ in } M,
%\end{aligned} \end{equation}
 \end{corollary}
%\begin{remark}
%The condition $\Gamma\neq\Gamma_n$ is crucial.
%\end{remark}

%We can check on $4$-manifolds, $|Ric_g|^2=|G_g|^2$, and then obtain Theorem \ref{thm0}.

\begin{remark}
Notice the rescaling invariance: $A_g^{\tau,\alpha}=A_{t g}^{\tau,\alpha}$ for all constants $t>0$. The constant $\Lambda_0$ in 
\eqref{key-assum1} can be chosen as $\Lambda_0=1$.
%hence $\delta=1$ in Theorems \ref{thm4}, \ref{thm-n-1}, \ref{thm-sigma2} and Corollaries \ref{coro2}, \ref{coro1}.
\end{remark}

\begin{remark}
In Theorem \ref{thm1}, $\tau$ can be replaced by a smooth function obeying \eqref{tau-alpha}.
\end{remark}

The article is organized as follows. 
In Section \ref{preliminaries} we collect some useful notation and formulas, and briefly discuss some results on the partial uniform ellipticity.
In Section \ref{section3} we obtain local estimates for gradient and Hessian by confirming the fully uniform ellipticity of a more general class of fully nonlinear elliptic equations which includes equation \eqref{main-equ1} as a special case. 
Using the local estimate, we obtain existence results of solutions for such equations when the background space is a compact manifold with boundary. %The existence result also applies to local $C^0$-estimate of approximate Dirichlet problems.
% that we will use in the proof of main result.
In Section \ref{section4} we derive local $C^0$-estimate for approximate solutions and 
complete the proof of our main existence theorem by using an approximate process.
% A key ingredient is local zero order estimate.
In Section \ref{section5},  by confirming the asymptotic property near infinity \eqref{key-assum1},
%with a restriction to decay ratio of prescribing curvature functions, 
 we obtain  on Euclidean spaces
the existence of smooth  complete metrics of prescribed curvature functions with a restriction to decay ratio. 
In Section \ref{section6} we briefly study the conformal deformation of $-A_g$ %but without presenting proof,
and then apply it to Einstein tensor and certain modified Schouten tensors by constructing certain cones of  type 2.
 
\section{Preliminaries}
\label{preliminaries}

\subsection{Notation}

Let $e_1,...,e_n$ be a local frame on $M$. We denote $g_{ij }= g(e_i,e_j), ({g^{ij} })= ({g_{ij}})^{-1}$. 
Under Levi-Civita connection $\nabla$ of $(M,g)$, 
$\nabla_{e_i}e_j=\Gamma_{ij}^k e_k$, and $\Gamma_{ij}^k$ denote the Christoffel symbols and $\Gamma_{ij}^k=\Gamma_{ji}^k$.  For simplicity we write
$\nabla_i=\nabla_{e_i}, \nabla_{ij}=\nabla_i(\nabla_j)-\Gamma_{ij}^k\nabla_k.
%\nabla_{ijk}=\nabla_i(\nabla_{jk})-\Gamma_{ij}^l\nabla_{lk}-\Gamma^l_{ik}\nabla_{jl}, \mbox{ etc}.
$
On $(M,g)$ one also defines the curvature tensor by $$R(X,Y)Z=-\nabla_X\nabla_YZ+\nabla_Y\nabla_X Z+\nabla_{[X,Y]}Z.$$ %The curvature coefficients are given as $R_{ijkl}=g(e_i,R(e_k,e_l)e_j)$. Then Ricci curvature is $R_{ij}=R^k_{ikj}$ where $R^{i}_{jkl}=g^{iq}R_{qjkl}.$

Under a locally unit orthogonal frame $e_1,\cdots, e_n$, 
i.e. $g_{ij}=\delta_{ij}$,  
\begin{equation}  \begin{aligned}
\,& Ric_g(e_i,e_i)=\sum_{j=1}^n g(R(e_i,e_j)e_i,e_j), 
\,& R_g=\sum_{i=1}^n Ric_g(e_i,e_i), \nonumber
 \end{aligned}  \end{equation}
thus one can check  
\begin{equation} \label{GSW1} \begin{aligned}
  Ric_g(e_i,e_i)=\frac{R_g}{2}-\frac{1}{2}\sum_{k,l\neq i} g(R(e_k,e_l)e_k,e_l), \mbox{  }  \forall 1\leq i\leq n.
 \end{aligned}  \end{equation}

\subsection{Some formulas for conformal change}
Under the conformal change $\tilde{g}=e^{2u}g$,
 \begin{equation}
 \begin{aligned}
 Ric_{\tilde{g}}=\,& Ric_g -\Delta u g -(n-2)\nabla^2u-(n-2)|\nabla u|^2g +(n-2)du\otimes du, \nonumber
 \end{aligned}
 \end{equation}
  \begin{equation}
 \begin{aligned}
 e^{2u}R_{\tilde{g}}=\,& R_g-2(n-1)\Delta u-(n-1)(n-2)|\nabla u|^2, \nonumber
 \end{aligned}
 \end{equation}
 where and hereafter $\Delta u$, $\nabla^2 u$ and $\nabla u$ are respectively the Laplacian, Hessian and gradient of $u$ with respect to $g$, $|\nabla u|^2=g^{ij}\nabla_i u \nabla_j u$.
 Thus
\begin{equation}
\label{conformal-formula1}
 \begin{aligned}
 A_{\tilde{g}}^{\tau,\alpha}
 =\,& A_{g}^{\tau,\alpha}
 +\frac{\alpha(\tau-1)}{n-2}\Delta u g-\alpha  \nabla^2 u
  +\frac{\alpha(\tau-2)}{2}|\nabla u|^2 g
 % \\\,&
  +\alpha  du\otimes du. %\nonumber
 \end{aligned}
 \end{equation}
  
 %\subsection{Existence of complete conformal metric on manifolds with boundary}
% \subsection{Loewner-Nirenberg and Aviles-McOwen's result}
\subsection{Existence of  complete conformal metric with scalar curvature $-1$}
 A  key tool %of this paper %for proving Theorem \ref{thm-c0-upper}
 is a result due to 
% Loewner-Nirenberg and Aviles-McOwen \cite{Aviles1988McOwen,Loewner1974Nirenberg}.
  Aviles-McOwen \cite{Aviles1988McOwen} who extended a theorem of Loewner-Nirenberg \cite{Loewner1974Nirenberg} from smooth bounded domains $\Omega\subset\mathbb{R}^n$ to general compact Riemannian manifolds with boundary.  
 Namely,
\begin{theorem}
%[\cite{Aviles1988McOwen,Loewner1974Nirenberg}]
\label{thm1-AM}
Let $(M,g)$ be a compact Riemannian manifold with smooth boundary, then
 $(M,g)$ has a complete conformal metric with scalar curvature $-1$.

\end{theorem}

\subsection{Partial uniform ellipticity revisited}
%It should be notable that the cone $\Gamma_n\subset\Gamma$ is not necessary in this subsection.
For $\sigma<\sup_\Gamma f$, we denote 
\begin{equation}
 \begin{aligned}
  \Gamma^\sigma=\{\lambda\in \Gamma: f(\lambda)>\sigma\}.
%\,& \partial \Gamma^\sigma=\{\lambda\in\Gamma: f(\lambda)=\sigma\}.
 % \mbox{ So } \Gamma\setminus\Gamma^{\sigma}=\{\lambda\in\Gamma: f(\lambda)\leq \sigma\}. \nonumber
 \end{aligned}
 \end{equation}
 So $\Gamma\setminus\Gamma^{\sigma}=\{\lambda\in\Gamma: f(\lambda)\leq \sigma\}$.
% We also denote $Df(\lambda)=(f_1(\lambda),\cdots,f_n(\lambda)).$
%The normal vector $\nu_\lambda$ of the level hypersurface $\partial\Gamma^\sigma$ is given by
%$$\nu_{\lambda}=Df(\lambda)$$

The following lemma is a refinement of Lemma 3.4 in \cite{yuan2019} (also Lemma 2.1 of \cite{yuan2020conformal}).
\begin{lemma} %[\cite{yuan2019}]
 \label{lemma3.4} 
 Let  $ \sup_{\partial \Gamma}f<\tau_0<\sup_{\Gamma}f$ be fixed.
 For the $f$ satisfying \eqref{concave}, and
 \begin{equation}
 \label{elliptic-weak}
 \begin{aligned}
\,& %Df(\lambda)\in\overline{\Gamma}_n,
f_i(\lambda)\geq0, \,&  \forall 1\leq i\leq n, \mbox{  }  \forall \lambda\in\Gamma,
 \end{aligned}
 \end{equation}
  the following two statements are equivalent each other:
%  \begin{itemize}
  \begin{equation}
\label{addistruc-weak}
\begin{aligned}
 \,& \lim_{t\rightarrow +\infty}f(t\lambda)>\tau_0, % \mbox{ for each } \sigma\leq \tau_0, 
 \,& \forall\lambda\in\Gamma.
\end{aligned}
\end{equation}
%Then for such $\tau_0$, we have
  \begin{equation}
 \label{yuan23}
 \begin{aligned}
\,& \sum_{i=1}^n f_i(\lambda)\mu_i>0, \,&  \forall \lambda\in\Gamma\setminus\Gamma^{\tau_0}, \mbox{  }  \forall \mu\in\Gamma.
 \end{aligned}
 \end{equation}
% \end{itemize}
  \end{lemma}

\begin{proof}

\eqref{addistruc-weak} $\Rightarrow$ \eqref{yuan23}:
The proof is different from that given in \cite{yuan2019}.
 For given $\lambda\in\Gamma\setminus\Gamma^{\tau_0}$ and $\mu\in\Gamma$, by the concavity we have
 \begin{equation}\label{concavity1}\begin{aligned}
 t\sum_{i=1}^n f_i(\lambda)\mu_i \geq \sum_{i=1}^n f_i(\lambda)\lambda_i
 + f(t\mu)-f(\lambda). \nonumber
\end{aligned}\end{equation}
Let $\mu=\lambda$ and $t\gg1$, we have $ \sum_{i=1}^n f_i(\lambda)\lambda_i>0$.
Thus \eqref{yuan23} holds.

\eqref{addistruc-weak} $\Leftarrow$ \eqref{yuan23}: Let $f(c(\tau_0)\vec{\bf 1})=\tau_0$. For any $\lambda\in\Gamma$, if $t$ is sufficiently large then $t\lambda-c(\tau_0)\vec{\bf 1}\in\Gamma$. So
$f(t\lambda)-\tau_0\geq \sum_{i=1}^n f_i(t\lambda)(t\lambda_i-c(\tau_0))>0$
which yields \eqref{addistruc-weak}.
\end{proof}

  Using Lemma \ref{lemma3.4},  as in \cite{yuan2020conformal}, we have the following results which slightly refine
 Lemma \ref{yuan-k+1} above and Proposition 8.1 of \cite{yuan2020conformal}, respectively. 
  \begin{proposition}
\label{yuan-k+1-weak}
%Let $f$ and $\Gamma$ be as above and $\kappa_\Gamma$ be as defined in Definition \ref{yuan-kappa}.
%Suppose $f$ satisfies \eqref{elliptic}, \eqref{concave} and \eqref{addistruc}.
In addition to \eqref{concave},  \eqref{elliptic-weak}, we assume \eqref{addistruc-weak} holds for 
some $ \sup_{\partial \Gamma}f<\tau_0<\sup_{\Gamma}f$.
Then for any $ \lambda\in \Gamma\setminus\Gamma^{\tau_0}$ with order $\lambda_1 \leq \cdots \leq\lambda_n$ and for 
$1\leq i\leq \kappa_\Gamma+1$,
   \begin{equation}
   \begin{aligned}
   f_{{i}}(\lambda) \geq   \vartheta_{\Gamma} \sum_{j=1}^{n}f_j(\lambda),  \nonumber
   %\mbox{ for all }  1\leq i\leq \kappa_\Gamma+1. 
    %\mbox{  } \cdots, \mbox{  } f_{{{1+\kappa_{\Gamma}}}}(\lambda) \geq  \vartheta_{\Gamma} \sum_{j=1}^{n}f_j(\lambda).
   \end{aligned}
   \end{equation}
   holds uniformly for a universally positive constant $\vartheta_{\Gamma}$ depending only on  
$\Gamma$.

%  In fact,  $f_{{1}}(\lambda) \geq   \frac{1}{n} \sum_{j=1}^{n}f_j(\lambda)$ and $0<\vartheta_{\Gamma}\leq\frac{1}{n}$. %and the specific choice for $\vartheta_{\Gamma}$ will be given in \eqref{theta1}.
%Here $\kappa_\Gamma$ is the constant defined as in Definition \ref{yuan-kappa}.

\end{proposition}

\begin{proof}
The proof is essentially the same as that of \cite{yuan2020conformal}, we present the detail for the choice of the constant $\vartheta_{\Gamma}$ in Lemma \ref{yuan-k+1} and Proposition \ref{yuan-k+1-weak} as well.

Let $ \lambda\in \Gamma\setminus\Gamma^{\tau_0}$ and we assume $\lambda_1 \leq \cdots \leq\lambda_n$.
If $\Gamma=\Gamma_n$, equivalently to $\kappa_\Gamma=0$, then the statement is true since 
$$f_1(\lambda)\geq \frac{1}{n}\sum_{i=1}^n f_i(\lambda).$$
For $\Gamma\neq\Gamma_n$, $\kappa_\Gamma\geq1$, let $\alpha_1, \cdots, \alpha_n$ be $n$ strictly positive constants such that 
     $$(-\alpha_1,\cdots,-\alpha_{\kappa_\Gamma}, \alpha_{\kappa_\Gamma+1},\cdots, \alpha_n)\in \Gamma.$$
   From \eqref{yuan23},
     \begin{equation}
 \label{good1-yuan}
\begin{aligned}
-\sum_{i=1}^{\kappa_\Gamma} \alpha_i f_i(\lambda)+\sum_{i=\kappa_\Gamma+1}^n \alpha_i f_i(\lambda)>0,
\end{aligned}
\end{equation}
which yields $f_{\kappa_\Gamma+1}(\lambda)>   \frac{\alpha_1}{\sum_{i=\kappa_\Gamma+1}^n \alpha_i}f_1(\lambda)$.
By \eqref{good1-yuan} and iteration,  
one derives
 \begin{equation}\label{theta1}
\begin{aligned}
 f_{\kappa_\Gamma+1}(\lambda)\geq\frac{\alpha_1}{(\sum_{i=\kappa_\Gamma+1}^n \alpha_i-\sum_{i=2}^{\kappa_\Gamma}\alpha_i)}f_1(\lambda). \nonumber
\end{aligned}
\end{equation}
From the discussion above, we see $\vartheta_\Gamma$ %in \eqref{k+1elliptic} 
can be achieved as $\vartheta_\Gamma=\frac{\alpha_1}{n(\sum_{i=\kappa_\Gamma+1}^n \alpha_i-\sum_{i=2}^{\kappa_\Gamma}\alpha_i)}.$
\end{proof}

\begin{remark}
%For the choice of the constant $\vartheta_{\Gamma}$ in Theorem \ref{yuan-k+1} and Proposition \ref{yuan-k+1-weak} as well, 
%please refer to Proposition 2.2 of \cite{yuan2020conformal}.  
Obviously, $0\leq\kappa_\Gamma\leq n-1$ and $0<\theta_\Gamma\leq \frac{1}{n}$; 
furthermore, $\vartheta_\Gamma=\frac{1}{n}$ and $\kappa_\Gamma=n-1$ cannot occur simultaneously. %simultaneity.
% (otherwise $f_i(\lambda)=\frac{1}{n}\sum_{j=1}^n f_j(\lambda)$).

\end{remark}

\begin{proposition}
%[\cite{yuan2020conformal}]
\label{prop1-operator}
%Let $\vartheta$ be as in Theorem \ref{yuan-k+1}.  
 
For the $f$ satisfying \eqref{concave}, \eqref{elliptic-weak} and \eqref{addistruc-weak} for some $\sup_{\partial \Gamma}f<\tau_0<\sup_\Gamma f$, then
 %there exists a universally positive constant $\vartheta$ (same as in Theorem \ref{yuan-k+1}) depending only on $n$ and $\Gamma$, such that 
   for each $\lambda\in \Gamma\setminus\Gamma^{\tau_0}$ we have
\begin{equation} 
\label{key1-yuan}
\begin{aligned}
f_i(\lambda)\geq \vartheta_\Gamma \sum_{j=1}^n f_j(\lambda)  \mbox{ if } \lambda_i\leq 0.
\end{aligned}
\end{equation}
%Here $\vartheta_\Gamma$ is the constant as in Theorem \ref{yuan-k+1}. 
Moreover,  \eqref{addistruc-weak} can be removed for $n=2$.
\end{proposition}
\begin{proof}
% [Proof of Proposition \ref{prop1-operator}]
 % The proof is based on Theorem \ref{yuan-k+1}.
%For simplicity $\kappa=\kappa_\Gamma$ for $\kappa_\Gamma$ defined in Definition \ref{yuan-kappa}.
Fix $\lambda=(\lambda_1,\cdots,\lambda_n)\in \Gamma\setminus\Gamma^{\tau_0}.$ Without loss of generality, we assume
$\lambda_n\geq \cdots\geq \lambda_1.$
 %Then the concavity yields $f_1(\lambda)\geq \cdots\geq f_n(\lambda).$
By the definition of $\kappa_\Gamma$, $\lambda_{\kappa_\Gamma+1}>0.$ Therefore, %if $\lambda_i\leq 0$ then
$i\leq \kappa_\Gamma \mbox{ if } \lambda_i\leq0,$
which deduces \eqref{key1-yuan}.
%Proposition \ref{prop1-operator} 
%So \eqref{key1-yuan} immediately follows from Theorem \ref{yuan-k+1}.
\end{proof}

Proposition \ref{prop1-operator} is very useful.
Together with 
$\sum_{i=1}^n f_i(\lambda)\geq \kappa(\tau_0)>0$,
 $\lambda\in \Gamma\setminus\Gamma^{\tau_0}$ (since $\sum_{i=1}^n f_i(\lambda)\lambda_i>0)$,
%\begin{equation} \label{sumfi1} \begin{aligned} \sum_{i=1}^n f_i(\lambda)\geq \kappa(\tau_0)>0, \mbox{ } \lambda\in \Gamma\setminus\Gamma^{\tau_0}  \mbox{ (since } \sum_{i=1}^n f_i(\lambda)\lambda_i>0), \nonumber \end{aligned}\end{equation}
this proposition confirms a key assumption which is used to derive (local and global) a priori estimates for solutions to
various partial differential equations from differential geometry. For more details we refer to \cite{SChen2009,Guan1991Spruck,LiYY1991,ShengUrbasWang-Duke,Trudinger90,Urbas2002}.

 \subsection{Proof of \eqref{positive1}}
We verify \eqref{positive1} that plays an important role in this paper.
%\begin{lemma}
% We have \eqref{positive1} given assumption \eqref{tau-alpha}.
%\end{lemma}
%\begin{proof}

 Case 1: $\alpha=-1$. As we assume in \eqref{tau-alpha},  $\tau<1$. Then $n\tau+2-2n<2-n<0$.
 
 Case 2: $\alpha=1$. We have $\tau>1+(n-2)(1-\kappa_\Gamma\vartheta_\Gamma)$ according to \eqref{tau-alpha}.
 Since $0<\vartheta_\Gamma\leq \frac{1}{n}$ and $0\leq\kappa_\Gamma\leq n-1$, we see $\tau>\frac{2n-2}{n}$. 
 Thus
  $n\tau+2-2n>0.$
%\end{proof}

\section{Local estimates and existence of solutions for some fully nonlinear equations on compact manifolds with boundary}
\label{section3}

Throughout this section, we assume that $(M,g)$ is a $n$-dimensional $(n\geq3)$ compact Riemannian manifold with smooth boundary 
$\partial M$, $\bar M=M\cup\partial M$, $A$ is a smooth symmetric $(0,2)$-type tensor on $\bar M$.
 %and with the Levi-Civita connection $\nabla$.
Let 
\begin{equation}
 \begin{aligned}
 V[u]=A+\Delta u g -\varrho(x)\nabla^2 u  + a(x)|\nabla u|^2 g + b(x)du\otimes du+c(x) L(du),
 \end{aligned}
 \end{equation}
where 
$\varrho(x)$, $a(x)$, $b(x)$ and $c(x)$ are smooth functions,  and 
$L(du)$ is a smooth symmetric $(0,2)$-type tensor depending linearly on $du$.
In addition, we replace \eqref{homogeneous-1} by
\begin{equation}
\label{homogeneous-gamma}
\begin{aligned}
f(t\lambda)=t^\gamma f(\lambda), \mbox{  } f(\lambda)>0, \mbox{ } f|_{\partial\Gamma}=0, 
 \end{aligned}
 \end{equation}
$\mbox{ for some } 0<\gamma\leq1,
 \mbox{ for any } \lambda\in\Gamma, \mbox{  } t>0.$
 
We consider the equation
\begin{equation}
\label{mainequ-2}
 \begin{aligned}
\,& F(V[u]):= f(\lambda(g^{-1}V[u]))=\psi e^{2\gamma u} \mbox{ in } M, \,& 0<\psi\in C^\infty(\bar M).
 \end{aligned}
 \end{equation}
 In order to study \eqref{mainequ-2} in the framework of elliptic equations, 
 we are going to look for solutions in class of \textit{admissible} functions satisfying 
\begin{equation}
\label{admissible-function}
 \begin{aligned}
\,& \lambda(g^{-1}V[u]) \in \Gamma \mbox{ in } \bar M, \,& u\in C^{2}(\bar M). %\nonumber
 \end{aligned}
 \end{equation}

\subsection{Local and boundary estimates for equation \eqref{mainequ-2}}

We have local and boundary estimates for admissible solutions for \eqref{mainequ-2} under the assumption that
\begin{equation}
\label{assumption-4}
\begin{aligned}
%\Gamma\neq \Gamma_n  (\mbox{i.e. }  \kappa_\Gamma\geq 1) \mbox{ and }
0<  \varrho(x)<\frac{1}{1-\kappa_\Gamma \vartheta_{\Gamma}}, \mbox{ or } \varrho(x)<0 \mbox{ in } \bar M.
%\mbox{  where $\vartheta_{\Gamma}$ is   as in Theorem \ref{yuan-k+1}.}
 \end{aligned}
\end{equation}
Together with  $0<\vartheta_\Gamma\leq \frac{1}{n}$, $0\leq\kappa_\Gamma\leq n-1$ this condition implies
 \begin{equation}
 \label{varrhon}
 \begin{aligned}
 \varrho(x)<n,
 \end{aligned}
 \end{equation}
 which will play an important role in the construction of locally upper barriers as we will show below.
\begin{theorem}
\label{thm1-local}
Suppose  \eqref{elliptic}, \eqref{concave}, \eqref{homogeneous-gamma} and \eqref{assumption-4} hold.
Let $B_r\subset M$ be a geodesic ball of radius of $r>0$.
Let $u\in C^4(B_r)$ be an admissible solution of \eqref{mainequ-2} in $B_r$, then 
\begin{equation}
 \begin{aligned}
 \sup_{B_{r/2}}(|\nabla u|^2 +|\nabla^2 u| )\leq  {C/r^2},
 \end{aligned}
 \end{equation}
 where $C$ is a uniformly positive constant depending on $|u|_{C^{0}(B_r)}$ and other known data in $B_r$.
\end{theorem}

\subsubsection{Fully uniform ellipticity}

The linearized operator $\mathcal{L}$ of \eqref{mainequ-2} at $u$ is given by
\begin{equation}
 \begin{aligned}
\mathcal{L}v
= \,& F^{ij}g_{ij} \Delta v -\varrho(x) F^{ij}\nabla_{ij}v +2a(x)g(\nabla u, \nabla v) F^{ij}g_{ij} 
\\\,& +b(x) F^{ij}(\nabla_i u\nabla_j v+\nabla_i v \nabla_j u) + c(x) F^{ij}(L(du))_{ij}\\
=\,&
(F^{pq}g_{pq}g^{ij}-\varrho(x)F^{ij})\nabla_{ij} v + \mbox{lower order terms}, \nonumber
 \end{aligned}
 \end{equation}
 where $F^{ij}=\frac{\partial F}{\partial V_{ij}}(V)$, $V=V[u]$.

 The eigenvalues of matrix $\{F^{pq}g_{pq}g^{ij}-\varrho(x)F^{ij}\}$ with respect to $\{g^{ij}\}$ are precisely 
 $$\sum_{j=1}^n f_{j}-\varrho(x) f_1, \cdots, \sum_{j=1}^n f_{j}-\varrho(x) f_n.$$
 
 Next, we prove that \eqref{mainequ-2} is fully uniform ellipticity, provided \eqref{assumption-4} holds. That is
 \begin{proposition}
\label{keykey}
There is  $\theta>0$
so that
  \begin{equation}
  \label{fully-uniform2}
 \begin{aligned}
\,&  \sum_{j=1}^n f_{j}-\varrho(x) f_i \geq \theta\sum_{j=1}^n f_j, \,& \forall i.
 \end{aligned}
 \end{equation}

\end{proposition}

\begin{proof}
 If $\varrho<0$, \eqref{fully-uniform2} holds for $\theta=1$. 
  In what follows we assume $0<\varrho(x)<\frac{1}{1-\kappa_\Gamma \vartheta_{\Gamma}}$.
As a corollary of Lemma \ref{yuan-k+1}, $\sum_{j=1}^n f_j-f_i\geq \kappa_\Gamma\vartheta_\Gamma \sum_{j=1}^n f_j$, i.e.,
\begin{equation}
\label{sumfi-3}
\begin{aligned}
\sum_{j=1}^n f_j 
 \geq
 \left(1-\varrho(x)(1-\kappa_\Gamma\vartheta_\Gamma)\right)\sum_{j=1}^n f_j+\varrho(x) f_i\nonumber
\end{aligned}
\end{equation}
for all $1\leq i\leq n$. 
 This proves \eqref{fully-uniform2}.
 
 \end{proof} 
 With Proposition \ref{keykey} above at hand,
 Theorem \ref{thm1-local} follows as a corollary of Theorems 5.1 and 5.3 of \cite{yuan2020conformal};
 furthermore, by Theorem 5.4 there, we obtain the following boundary estimate.
  \begin{theorem}
 \label{thm2-bdy}
Suppose  \eqref{elliptic}, \eqref{concave},  \eqref{homogeneous-gamma},    \eqref{assumption-4} hold.
 Let $u\in C^3(M)\cap C^2(\bar M)$ be an admissible solution to equation \eqref{mainequ-2} with 
 $u=\varphi$ on $\partial M$, $\varphi\in C^3(\bar M)$. Then $$\sup_{\partial M}|\nabla^2 u|\leq C$$ holds for 
 a uniformly positive constant $C$ depending on  $ \left( \inf_{\bar M} (1-\varrho(x)(1-\kappa_\Gamma\vartheta_\Gamma)\right)^{-1}$,  $|u|_{C^1(\bar M)}$, $|\varphi|_{C^3(\bar M)}$, and other known data. 
 
 \end{theorem}

\subsection{Existence results}

%The starting metric is an admissible function $w\in C^2(\bar M)$.
In this subsection, we  assume there exists a $C^2$-admissible function $\underline{w}$ satisfying \eqref{admissible-function}; furthermore we assume  there is some positive constant $\beta$ such that
\begin{equation}
\label{yuan-canshu1}
 \begin{aligned}
 \frac{1}{\beta}+a(x)>0, 
 \end{aligned}
 \end{equation}
 \begin{equation}
 \label{yuan-canshu2}
 \begin{aligned}
 b(x)-\frac{\varrho(x)}{\beta}\geq0, \mbox{ or }  %b(x)-\frac{\varrho(x)}{\beta}<0, 
 a(x)+b(x)+\frac{1-\varrho(x)}{\beta}>0
 \end{aligned}
 \end{equation}
 which are used to construct locally lower/upper barriers.
 In particular, such two conditions hold for $0<\beta\ll1$, given $\varrho(x)<0$ or $0<\varrho(x)<1$.
  We prove
  \begin{theorem}
  \label{thm-existence-general00}
  Suppose  \eqref{elliptic}, \eqref{concave}, \eqref{homogeneous-gamma},
   \eqref{assumption-4}, \eqref{yuan-canshu1},  \eqref{yuan-canshu2} hold.
For any $\varphi\in C^\infty(\partial M)$, there is a unique smooth admissible function $u$  to solve \eqref{mainequ-2} with %boundary value condition 
 \begin{equation}
 \label{bdy-condition2}
 \begin{aligned}
 u|_{\partial M}=\varphi,
 \end{aligned}
 \end{equation}
provided that there exists a $C^2$-admissible function $\underline{w}$ satisfying \eqref{admissible-function}.

\end{theorem}

 \subsubsection{$C^0$-estimate}
% We first prove two lemmas  (comparison principle) by using maximum principle.
 \begin{lemma}
%[Comparison principle]
 \label{lemma-mp}
 Let $w, v\in C^2(\bar M)$ be admissible functions with
  \begin{equation}
 \begin{aligned}
\,& f(\lambda(V[w]))\geq \psi e^{2\gamma w}, \,& f(\lambda(V[v]))\leq \psi e^{2\gamma v} \mbox{ in } M, \nonumber 
 \end{aligned}
 \end{equation}
  \begin{equation}
 \begin{aligned}
(w-v)|_{\partial M} \leq 0, \nonumber
 \end{aligned}
 \end{equation}
 then $w-v\leq 0$ in $M$.
 \end{lemma}
 \begin{proof}
 Suppose that there is an interior point $x_0\in M$ such that $0<(w-v)(x_0)=\sup_M (w-v)$, % By the maximum principle, 
 therefore at $x_0$,
    \begin{equation}
 \begin{aligned}
\,& \nabla^2 v \geq \nabla^2 w, \,& \nabla v=\nabla w. \nonumber 
 \end{aligned}
 \end{equation}
 Combining with Proposition \ref{keykey},
   \begin{equation}
 \begin{aligned}
\,& f(\lambda(V[v]))\geq f(\lambda(V[w])), \nonumber 
 \end{aligned}
 \end{equation}
which further yields $v(x_0)\geq w(x_0)$. This contradicts to  $w(x_0)>v(x_0)$.
 \end{proof}
 
 \begin{lemma}
 \label{lemma-c0general}
 Let $u\in C^2(\bar M)$ be an admissible solution to Dirichlet problem  \eqref{mainequ-2} and \eqref{bdy-condition2},
 then
  \begin{equation}
 \begin{aligned}
  \inf_M (u-\underline{w}) \geq\,&
  \min\left\{\inf_{\partial M}(\varphi-\underline{w}), \frac{1}{2\gamma}\inf_M \log\frac{f(\lambda(g^{-1}V[\underline{w}]))}{\psi} \right\}, 
  \\
 \sup_M (u-\underline{w}) \leq \,& 
 \max\left\{\sup_{\partial M}(\varphi-\underline{w}), 
 \frac{1}{2\gamma}\sup_M \log\frac{f(\lambda(g^{-1}V[\underline{w}]))}{\psi} \right\}.
 \end{aligned}
 \end{equation}
 
 \end{lemma}
 
 \begin{proof}
 The proof is based on maximum principle. We omit the detail here.
 \end{proof}
 
 \subsubsection{Local barriers and boundary estimate for gradient}
 
 Let $\rho(x)$ be the distance function from $x$ to boundary $\partial M$ with respect to $g$, $$\Omega_\delta=\{x\in M: \rho(x)<\delta\}.$$
Notice $\rho$ is smooth in $\Omega_\delta$ when $0<\delta\ll1$, and $|\nabla \rho|=1$ on ${\partial M}$.

 \subsubsection*{Locally lower barrier}
We use $\rho$ to construct local barriers.
Let $$w=\beta\log \frac{\delta^2}{\delta^2 + \rho}.$$
Then $e^{2\gamma w}=(\frac{\delta^2}{\delta^2+\rho})^{2\beta\gamma}.$
First, a straightforward computation gives
 \begin{lemma}
 \begin{equation}
 \begin{aligned}
  V[w]=\,& A+ \frac{\beta(1+a\beta)}{(\delta^2+\rho)^2}|\nabla \rho|^2 g+\frac{\beta(b\beta-\varrho)}{(\delta^2+\rho)^2}d\rho\otimes d\rho \\\,&
  +cL(\frac{-\beta}{\delta^2+\rho}d\rho)-\frac{\beta}{\delta^2+\rho}(\Delta \rho g-\varrho \nabla^2 \rho), \nonumber
 \end{aligned}
 \end{equation}
 \begin{equation}
 \begin{aligned}
 V[w+\varphi]=\,&
 V[w]+c(L(d(w+\varphi))-L(dw))+(\Delta \varphi g-\rho\nabla^2  \varphi) \\ \nonumber
 \,&
 +a\left(|\nabla \varphi|^2 -\frac{2\beta}{\delta^2+\rho}g(\nabla \varphi,\nabla \rho) \right)g
 +bd\varphi\otimes d\varphi \\
 \,& -\frac{\beta b}{\delta^2+\rho}(d\varphi\otimes d\rho+d\rho\otimes d\varphi)
 \end{aligned}
 \end{equation}
 \end{lemma}
 
 Using this lemma, if \eqref{yuan-canshu1} and \eqref{yuan-canshu2}  hold then there is a positive constant $c_0$ such that for small 
 $\delta$, 
 \begin{equation}
 \begin{aligned}
\,& V[w]\geq \frac{c_0}{(\delta^2+\rho)^2}|\nabla \rho|^2 g, \,& V[w+\varphi]\geq \frac{c_0}{2(\delta^2+\rho)^2}|\nabla \rho|^2 g \mbox{ on } \Omega_\delta, \nonumber
 \end{aligned}
 \end{equation}
 and so
  \begin{equation}
 \begin{aligned}
 f(\lambda(g^{-1}V[w+\varphi]) \geq \psi e^{2\gamma(w+\varphi)}\mbox{ in } \Omega_\delta. \nonumber
 \end{aligned}
 \end{equation}

 One can check $(w+\varphi-u)|_{\partial M}=0$, and $$(w+\varphi-u)|_{\rho=\delta}=\beta\log\frac{\delta}{1+\delta}+(\varphi-u)|_{\rho=\delta}<0 \mbox{ if }\delta\ll1.$$ Here we use Lemma \ref{lemma-c0general}.
  Now we have a local lower barrier  $w+\varphi$ near boundary.

 \subsubsection*{Locally upper barrier}
 
 Let $v=\beta' \log(1+\frac{\rho}{\delta^2})+\varphi.$
  \begin{equation}
 \begin{aligned}
 \mathrm{tr} (g^{-1}V[v])=\,&
 \frac{\beta' |\nabla \rho|^2}{(\delta^2+\rho)^2}\left\{-n+\varrho+(na+b)\beta'\right\}
 + \frac{\beta'(n-\varrho)}{\delta^2+\rho} \Delta\rho \\
 \,&
 +\frac{2(na+b)\beta'}{\delta^2+\rho} g(\nabla\rho,\nabla\varphi)+c\mathrm{tr} (g^{-1}L(d\varphi+\frac{\beta'}{\delta^2+\rho}d\rho))
 \\ \,&
+ (n-\varrho)\Delta \varphi+(na+b)|\nabla \varphi|^2+\mathrm{tr}(g^{-1} A).  \nonumber
 \end{aligned}
 \end{equation}
 By \eqref{varrhon}, $-n+\varrho+(na+b)\beta'\leq -\frac{n-\varrho}{2}<0$ if $0<\beta'\ll1$. Note $|\nabla\rho|=1$ on $\partial M$. So we can choose $\delta$ small sufficiently so that
 $ \mathrm{tr}(g^{-1}V[v])\leq0 \mbox{ in } \Omega_\delta.$
 %Here we also use $|\nabla\rho|=1$ on $\partial M$.
  It is easy to see $(v-u)|_{\partial M}=0$. On $\{\rho=\delta\}$, by Lemma \ref{lemma-c0general} again, 
 $v-u=\beta' \log(1+\frac{1}{\delta})+\varphi-u>0$ for $0<\delta\ll1$.
 
 Therefore, near the boundary we now obtain a locally upper barrier $$v=\beta' \log(1+\frac{\rho}{\delta^2})+\varphi.$$
As a consequence, we obtain the gradient estimate at boundary $\sup_{\partial M} |\nabla u| \leq C.$
% \begin{equation} \begin{aligned}
%\sup_{\partial M} |\nabla u| \leq C. \nonumber
% \end{aligned} \end{equation}
 
By standard continuity method, we can prove Theorem \ref{thm-existence-general00}.
 The conformal metrics with prescribed boundary metric is also obtained.
 
  \begin{theorem}
 \label{thm-existence-general}
 %Let $(M,g)$ be a $n$-dimensional compact Riemannian manifold with smooth boundary, 
 
 Suppose \eqref{elliptic}, \eqref{concave}, \eqref{homogeneous-1},   
 \eqref{tau-alpha}  hold, and  $0<\psi\in C^\infty(\bar M)$. If, in addition, there is $\underline{u}\in C^2(\bar M)$ so that $\underline{g}=e^{2\underline{u}}g$ obeys
 \begin{equation}\label{admissible-23}\begin{aligned}
 \lambda(\underline{g}^{-1}A_{\underline{g}}^{\tau,\alpha})\in \Gamma \mbox{ in } \bar M.
 \end{aligned}\end{equation} 
 %$(M,g)$ supposes a $C^2$-\textit{admissible conformal} metric satisfying \eqref{admissible-metric1}. 
 For a Riemannian metric $h$ on $\partial M$ which is conformal to $g|_{\partial M}$, there is a smooth conformal 
 metric $\tilde{g}$ with prescribed boundary condition $\tilde{g}|_{\partial M}=h$ to satisfy \eqref{main-equ1}.   
 \end{theorem}
 \begin{proof}
 It requires to verify  \eqref{yuan-canshu1}-\eqref{yuan-canshu2}
given assumption \eqref{tau-alpha} holds.
 In our case $\varrho=\frac{n-2}{\tau-1}$, 
$a=\frac{(n-2)(\tau-2)}{2(\tau-1)}$, $b=\frac{n-2}{\tau-1}$, $c=0$. First, if $\tau<1$ and $\alpha=-1$ then $\varrho<0$ and so \eqref{yuan-canshu1} and \eqref{yuan-canshu2} hold for $0<\beta\ll1$; while for $\tau>1+(n-2)(1-\kappa_\Gamma\vartheta_\Gamma)$ and $\alpha=1$,   
 we can verify \eqref{yuan-canshu1} and \eqref{yuan-canshu2}  hold for $\beta=1$.  Here we use \eqref{positive1}.
%We already verified \eqref{yuan-canshu1}-\eqref{yuan-canshu2}. 
%Theorem \ref{thm-existence-general} then follows from Theorem \ref{thm-existence-general00}.
 \end{proof}

\subsection{Complete conformal metrics on manifolds on boundary}
%Moreover, %following the outline of proof of Theorems 1.1 and 1.2 in \cite{yuan2020conformal}, 
%we can generalize Theorems 1.1 and 1.7 of \cite{yuan2020conformal} to more general cases. 
The following theorem provides a straightforward approach to local bound of solutions from above.
 \begin{theorem}
 \label{existence1-compact}
 %Let $(M,g)$ be a compact Riemannian manifold with smooth boundary $\partial M$ and
Suppose \eqref{elliptic}, \eqref{concave}, \eqref{homogeneous-1}, \eqref{tau-alpha}, \eqref{admissible-23} hold and that $0<\psi\in C^\infty(\bar M)$.
Then there exists at least one smooth complete conformal 
 metric ${g}_\infty=e^{2u_\infty}g$ %that are conformal to $g$, 
 satisfying \eqref{main-equ1}. %and  $\lambda({g}_\infty^{-1}A_{{g}_\infty})\in \Gamma$ in $M$.
 
 %In addition, if $\psi|_{\partial M}\equiv 1$ then there is a unique smooth complete conformal  metric $g_\infty=e^{2u_\infty}g$ satisfying  \eqref{equ-6} and
% \begin{equation}\label{asymptotic-rate0-general}\begin{aligned}
%\lim_{x\rightarrow\partial M} (u_\infty(x)+\log {\rho}(x))=\frac{1}{2}\log \frac{\alpha(n\tau+2-2n)}{2(n-2)}. \nonumber
%\end{aligned}\end{equation} 
%As denoted above, $\rho(x)$ is the distance function to boundary with respect to $g$.
\end{theorem}
 
%Since the proof follows the outline of proof of Theorems 1.1 and 1.2 in \cite{yuan2020conformal},  we omit the details here.
\begin{proof}
%Without loss of generality, $\underline{u}\equiv0$. %Since $\lambda(g^{-1}A_{g}^{\tau,\alpha})\in \Gamma$, then $R_g<0$.

 From Theorem \ref{thm-existence-general} there exists a $u_k\in C^\infty(M)$ satisfying
\begin{equation}
\label{dirichlet-equk}
\begin{aligned}
\,& f(\lambda(g^{-1}A_{\tilde{g}_k}^{\tau,\alpha}))=\psi e^{2u_k}, \mbox{ } \tilde{g}_k=e^{2u_k}g \mbox{ in } M, \,& u_k=\log k \mbox{ on } \partial M.
\end{aligned}
\end{equation}
It follows from \eqref{key1-main} that 
\begin{equation}
\begin{aligned}
n\psi e^{2u_k}\leq \mathrm{tr}(g^{-1}A_{\tilde{g}_k}^{\tau,\alpha})=\frac{\alpha(n\tau+2-2n)}{2(n-1)(n-2)}(2(n-1)\Delta {u}_k+ (n-1)(n-2)|\nabla {u}_k|^2-R_g). \nonumber
\end{aligned}
\end{equation}
Theorem \ref{thm1-AM} yields that there is a smooth function $\tilde{u}\in C^\infty(M)$ so that
\begin{equation}
\label{AM-equ-0}
%\left\{
\begin{aligned}
\,& 2(n-1)\Delta \tilde{u}+ (n-1)(n-2)|\nabla \tilde{u}|^2-R_g=e^{2\tilde{u}} \mbox{ in } M, 
\,& \lim_{x\rightarrow \partial M} \tilde{u}(x)=+\infty. 
\end{aligned}
%\right.
\end{equation}
The maximum principle yields
 \begin{equation}
 \label{uniform-c0}
\begin{aligned}
u_k\leq u_{k+1}\leq \tilde{u} +\frac{1}{2}\log\frac{\alpha(n\tau+2-2n)}{2n(n-1)(n-2)\inf_M\psi} \mbox{ in } M \mbox{ for } k\geq 1.  %\nonumber
\end{aligned}
\end{equation}
%which shows that the function
 %Let $u_\infty(x):=\lim_{k\rightarrow+\infty}u_k(x)$ is well defined for all $x\in M$. 
 As a result, for any compact subset $K\subset\subset M$,
  $$|u_k|_{C^0(K)}\leq C_1(K) \mbox{ for all } k, \mbox{ for $C_1(K)$  being independent of $k$}.$$

Given any compact subset $K\subset\subset M$, we choose a compact subset $K_1\subset\subset M$ such that
$K\subset\subset K_1$.
By Theorem \ref{thm1-local}
and 
Evans-Krylov theorem, 
 $$|u_k|_{C^{2,\alpha}(K)}\leq C_2(K,K_1) \mbox{ for all } k, \mbox{ where $C_2(K,K_1)$ depends not on $k$}.$$
 %Combining with diagonal argument
Combining with Schauder theory, let $k\rightarrow +\infty$,
 there is a smoothly  admissible function $u_\infty$ in $M$  to solve
\begin{equation}
\label{dirichlet-equ-infty}
\begin{aligned}
 f(\lambda(g^{-1}A_{g_\infty}^{\tau,\alpha}))=\psi e^{2{u_\infty}}, \mbox{  } g_\infty=e^{2u_\infty}g \mbox{ in } M, 
\mbox{  } \lim_{x\rightarrow\partial M}u_\infty(x)=+\infty.  \nonumber
\end{aligned}
\end{equation}
Indeed, $u_\infty(x)=\lim_{k\rightarrow+\infty}u_k(x)$ for all $x\in M$. By \eqref{uniform-c0}, one has in $M$
 \begin{equation}
 \label{uniform-c0-123}
\begin{aligned}
u_\infty\leq \tilde{u} +\frac{1}{2}\log\frac{\alpha(n\tau+2-2n)}{2n(n-1)(n-2)\inf_M\psi}.  %\nonumber
\end{aligned}
\end{equation}

Next we prove $g_\infty$ is complete.
Let  $h_k=h_k(\rho)= \log\frac{k\delta^2}{k\rho+\delta^2}$.  
The straightforward computation gives $h_k''=h_k'^2=\frac{k^2}{(k\rho+\delta^2)^2}$ and
\begin{equation}
\begin{aligned}
\,& \frac{\alpha(\tau-1)}{n-2}\Delta h_k g -\alpha\nabla^2 h_k +\frac{\alpha(\tau-2)}{2}|\nabla h_k|^2 g+\alpha dh_k\otimes dh_k 
\\ =\,&
 \frac{\alpha(n\tau+2-2n)}{2(n-2)} h_k'^2 |\nabla\rho|^2g
+\alpha h'_k  (  \frac{\tau-1}{n-2}\Delta \rho g-\nabla^2 \rho). \nonumber
\end{aligned}
\end{equation}
Let $\hat{g}_k=e^{2h_k}g$ near boundary. 
Then, near the boundary $\Omega_\delta$, $0<\delta\ll1$, $ A_{\hat{g}_k}^{\tau,\alpha}\geq c_0 h_k'^2 g$ for some constant $c_0>0$ (here we use \eqref{positive1} and $|\nabla\rho|=1$ on $\partial M$), and 
\begin{equation}
\begin{aligned}
f(\lambda(g^{-1}A_{\hat{g}_k}^{\tau,\alpha})) \geq \frac{ c_0 k^2}{(k\rho+\delta^2)^2}\geq
 \frac{ \psi \delta^4 k^2}{(k\rho+\delta^2)^2}=\psi e^{2h_k}. \nonumber
\end{aligned}
\end{equation}
On $\{\rho=\delta\}$, $h_k(\delta)=\log\frac{k\delta}{k+\delta}\rightarrow-\infty$ as $\delta\rightarrow0$.
Again, maximum principle yields $u_k\geq \log\frac{k\delta^2}{k\rho+\delta^2}$ near boundary.
Thus $u_\infty +\log\rho\geq -C_0$  for some $C_0$
near the boundary. The metric $g_\infty=e^{2u_\infty}g$ is complete.

\end{proof}

\section{Local zero order estimate and  proof of main results}
\label{section4}

  Let $\{M_k\}_{k=1}^{+\infty}$ be an exhaustion domains of $M$ with
  \begin{equation}
 \begin{aligned}
  M=\cup_{k=1}^{\infty} M_k, \mbox{  }  \bar M_k =M_k\cup\partial M_k,\mbox{  }  \bar M_k\subset\subset M_{k+1}, \nonumber
 \end{aligned}
 \end{equation}
 $$ \bar M_k  \mbox{ is a compact $n$-manifold with smooth boundary}.$$
 
  Without loss of generality, we can assume that the admissible complete conformal metric $\underline{g}=e^{2\underline{u}}g$
  satisfying assumption \eqref{key-assum1} is smooth, i.e. $\underline{u}\in C^\infty(M)$. 
 
 Let's consider a sequence of approximate Dirichlet problems 
 \begin{equation}
  \label{approximate-DP1}
 \begin{aligned}
f(\lambda({g}^{-1}A_{\tilde{g}}^{\tau,\alpha}))=\psi e^{2 u} \mbox{ in } M_k, \mbox{   } \tilde{g}=e^{2u}, \mbox{  }  
u=\underline{u} \mbox{ on } \partial M_k.
 \end{aligned}
 \end{equation}
 where %$A_{\tilde{g}}^{\tau,\alpha}=A_{g}^{\tau,\alpha} +\frac{\alpha(\tau-1)}{n-2}\Delta u g-\alpha  \nabla^2 u  +\frac{\alpha(\tau-2)}{2}|\nabla u|^2 g+\alpha  du\otimes du.$
  $A_{\tilde{g}}^{\tau,\alpha}$ obeys the formula \eqref{conformal-formula1} under the conformal change.
 According to Theorem \ref{thm-existence-general},  for each $k$, there is a unique smooth admissible conformal metric 
 $$g_k=e^{2u_k}g$$ to satisfy \eqref{approximate-DP1} on $M_k$.
 
 \begin{remark}
If $\underline{u}$ is only $C^2$, then on each $\partial M_k$, the boundary value condition
$(u-\underline{u})|_{\partial M_k}=0$ should be replaced by
$(u-\underline{u}_k)|_{\partial M_k}=0$, where $\underline{u}_k$ 
is a smooth function on $\bar M_k$ and satisfies
$$ |\underline{u}-\underline{u}_k|_{C^{2}(\bar M_k)}\leq \frac{1}{1+k},$$ %and
 \begin{equation}
 \begin{aligned}
\,&   f( \lambda(\underline{g}_k^{-1}A_{\underline{g}_k}^{\tau,\alpha}))\geq \frac{k}{1+k} f (\lambda(\underline{g}^{-1}A_{\underline{g}}^{\tau,\alpha})), \,&  
 \underline{g}_k=e^{2\underline{u}_k}g, \mbox{  }
 \lambda(\underline{g}_k^{-1}A_{\underline{g}_k}^{\tau,\alpha})\in \Gamma \mbox{ in } \bar M_k.  \nonumber
 \end{aligned}
 \end{equation}
 \end{remark}
  
  \subsection{Local zero order estimate}

To complete the proof of main results, it suffices to prove local zero order estimate for approximate Dirichlet problems \eqref{approximate-DP1}.

\begin{theorem}
\label{thm-c0-upper}
Let $K\subset\subset M$ be a compact subset of $M$, and assume $K\subset\subset M_m$ for some $m\in \mathbb{N}$.
Then there is a uniformly positive constant $C$ depending on $K$ and $M_m$ such that for each $u_k$ $(k\geq m)$ solving 
\eqref{approximate-DP1} on $M_k$, we have
\begin{equation}
 \begin{aligned}
 \sup_{K} u_k \leq C.  \nonumber
 \end{aligned}
 \end{equation}

\end{theorem}

\begin{proof}
%[Proof of Theorem \ref{thm-c0-upper}]

Here we present two proofs.

\noindent {\bf First proof}. According to Theorem \ref{existence1-compact}, there is a %unique 
%smooth function
 $w_m\in C^\infty(M_m)$ satisfying
\begin{equation}
%\left\{
 \begin{aligned}
 f(\lambda(g^{-1}A_{\tilde{g}_m}^{\tau,\alpha}))=\psi e^{2w_m},  \mbox{  } \tilde{g}_m=e^{2w_m}g, \mbox{  }  \lambda(g^{-1}A_{\tilde{g}_m}^{\tau,\alpha})\in \Gamma \mbox{ in } M_m, \mbox{  } 
 \lim_{x\rightarrow\partial M_m} w_m(x)=+\infty.  \nonumber
 \end{aligned}
% \right.
 \end{equation}
Applying the maximum principle, %(Lemma \ref{lemma-mp}),
 for any $k\geq m$, the (admissible) solution $u_k$ to \eqref{approximate-DP1} shall satisfy
\begin{equation}
 \begin{aligned}
u_k\leq w_m \mbox{ in } M_m,  \nonumber
 \end{aligned}
 \end{equation}
which then completes the first proof of Theorem \ref{thm-c0-upper}.
 
\noindent {\bf Second proof}.
A key ingredient is \eqref{key1-main}.  %i.e. $\sum_{i=1}^n \lambda_i \geq n f(\lambda)$.
 As a result, for all $k$, in  $M_k$
 \begin{equation}
 \begin{aligned}
 \mathrm{tr} (g^{-1}A_{g_k}^{\tau,\alpha})=\,&
 \frac{\alpha(n\tau+2-2n)}{2(n-1)(n-2)}
 \left\{2(n-1)\Delta {u}_k+ (n-1)(n-2)|\nabla {u}_k|^2-R_{g}\right\} 
% \\
  \geq n\psi e^{2u_k}. \nonumber
 \end{aligned}
 \end{equation}

Another key tool is Theorem \ref{thm1-AM}, which yields that there is a smooth function $\tilde{u}_m\in C^\infty(M_m)$ so that
\begin{equation}
\label{AM-equ}
%\left\{
\begin{aligned}
 2(n-1)\Delta \tilde{u}_m+ (n-1)(n-2)|\nabla \tilde{u}_m|^2-R_g=e^{2\tilde{u}_m} \mbox{ in } M_m,   \mbox{  } 
 \lim_{x\rightarrow \partial M_m} \tilde{u}_m(x)=+\infty. \nonumber
\end{aligned}
%\right.
\end{equation}
%Let $\tilde{u}_m$ be the solution to \eqref{AM-equ}.
  We observe that  for all $k\geq m$ the solution $\tilde{u}_m$ %to \eqref{AM-equ}
may provide a supersolution of Dirichlet problem \eqref{approximate-DP1} on $M_k$. 
%In addition, $$\tilde{u}_m\geq -C_m \mbox{ for some } C_m.$$
   The maximum principle gives for each $k\geq m$,
  \begin{equation}
 \begin{aligned}
 u_k \leq \tilde{u}_m+\frac{1}{2}\log\frac{\alpha(n\tau+2-2n)}{2n(n-1)(n-2)\inf_{M_m}\psi} \mbox{ in } M_m.  \nonumber
 \end{aligned}
 \end{equation}
 The second proof is complete.
 
% We remark furthermore that the first proof indeed depends on \eqref{admissible-metric1}, while the second one does not.
 \end{proof}
 
 From assumption \eqref{key-assum1} there is a uniformly positive constant $\Lambda_1$ such that 
\begin{equation}
\label{key-assum2}
 \begin{aligned}
\,&  f(\lambda({g}^{-1}A_{\underline{g}}^{\tau,\alpha})) \geq \Lambda_1 \psi e^{2 \underline{u}} \mbox{ in } M \mbox{ for some $\Lambda_1>0$}, \,& \underline{g}=e^{2\underline{u}}g. %\nonumber
 \end{aligned}
 \end{equation} 
 By comparison principle, we can derive $ \inf_{M_k} (u_k-\underline{u}) \geq  \min\left\{0, \frac{1}{2} \log\Lambda_1\right\}.$
 That is
 \begin{theorem}
 \label{thm-c0-lower}
 For the solution $u_k$ to \eqref{approximate-DP1}, we have $$u_k\geq \underline{u}+\frac{1}{2}\min\{0,  \log\Lambda_1 \} \mbox{ in } M_k.$$
 
\end{theorem}
%\begin{proof}

%By Lemma \ref{lemma-c0general} we can derive

% \begin{equation} \begin{aligned}  \inf_{M_k} (u_k-\underline{u}) \geq\,&  \min\left\{0, \frac{1}{2} \log\Lambda_1\right\}. \nonumber
% \end{aligned}\end{equation}
% \end{proof}

\subsection{Completion of proof of Theorem \ref{thm1}}

Let $K\subset\subset K_1\subset\subset M$ be two distinct  compact subsets of $M$, let $u_k$ be the admissible solution to 
Dirichlet problem \eqref{approximate-DP1} on $M_k$. 
For $K_1$ there is $k_0$ such that $K_1\subset\subset M_{k_0}$. By Theorems \ref{thm-c0-upper} and
\ref{thm-c0-lower}, there is a uniformly positive constant $C_0$ depending not on $k$ such that for any $k\geq k_0$,
\begin{equation}
 \begin{aligned}
 \sup_{K_1}|u_k|\leq C_0, \nonumber
 \end{aligned}
 \end{equation}
 then according to Theorem \ref{thm1-local}
% Combining with Theorem \ref{thm1-local} on local estimates, we derive there exists a uniformly positive constant $C_1$ depending on $K_1$, $K$ but not on $k$ such that for any  
 \begin{equation}
 \begin{aligned}
\,& |u_k|_{C^2(K)} \leq C_1, \,&\forall k\geq k_0 \nonumber
 \end{aligned}
 \end{equation}
 holds for a uniformly positive constant $C_1$ depending  not on $k$. 
 The Evans-Krylov theorem \cite{Evans82,Krylov83} and classical Schauder theory give
  \begin{equation}
 \begin{aligned}
\,& |u_k|_{C^{l,\alpha}(K)} \leq C_l=C_l(K,K_1), \,& \forall k\geq k_0, \mbox{  } l\geq 2, \mbox{ for some } 0<\alpha<1. \nonumber
 \end{aligned} 
 \end{equation}
 By diagonal process, we obtain a desired solution $u_\infty\in C^\infty(M)$ to equation \eqref{main-equ1}.
 
 From Theorem \ref{thm-c0-lower},
 $$u_\infty\geq \underline{u}+\frac{1}{2}\min\{0,  \log\Lambda_1 \} \mbox{ in } M.$$
 Combining with the completeness of $\underline{g}=e^{2\underline{u}}g$, we know $g_\infty=e^{2u_\infty}g$  is complete.
  This completes the proof of Theorem \ref{thm1}.
 
 \vspace{1mm}
 The method also works for prescribed scalar curvature equation. %Theorem \ref{thm-scalarcurvature}
\begin{theorem}
\label{thm-scalarcurvature}
Assume $R_g<0$.
 Let $\psi$ be a smooth positive function which is bounded from above in terms of $-R_g$, i.e. there is a constant $\delta>0$ such that 
 $0<\psi\leq -\delta R_g$. Then there exists a smooth complete conformal metric $\tilde{g}$ with $R_{\tilde{g}}=-\psi.$
 \end{theorem}
 
 \section{Complete conformal metrics on Euclidean spaces}
 \label{section5}

 There are two fundamental models of complete manifolds: standard hyperbolic space  $(\mathbb{H}^n, g_{-1})$ and flat Euclidean space $(\mathbb{R}^n, g_{0})$.
 
 In this section we deal with the latter one, and 
 the purpose is to prove the existence of complete metric
  that is conformal to the standard Euclidean metric.
 The main ingredient is to construct certain admissible complete smooth conformal metrics on $\mathbb{R}^n$ which then confirms assumption \eqref{key-assum1} of Theorem \ref{thm1}.

  Let $(x_1,\cdots,x_n)$ be the standard coordinate systems, and $ r^2=|x|^2=\sum_{i=1}^n x_i^2.$
  Let $h$ be a smooth radial symmetric function: 
  \begin{equation}
   \label{function-h}
 \begin{aligned}
h=\beta\log(1+r^2) \mbox{ for } \beta>0. \nonumber
 \end{aligned}
 \end{equation}

 Since $g_0$ is flat, $A_{g_0}^{\tau,\alpha}\equiv0$.  Let $g=e^{2h}g_0$, then $$A_{{g}}^{\tau,\alpha}= 
 \frac{\alpha(\tau-1)}{n-2}\Delta h g-\alpha D^2 h
  +\frac{\alpha(\tau-2)}{2}|Dh|^2 g
  +\alpha  dh\otimes dh.$$
   Some simple computations (with respect to the flat metric $g_0$) give
  \begin{equation}
 \begin{aligned}
\,& |D r^2|^2=4r^2, \,& D^2 r^2=2g_0, \nonumber
 \end{aligned}
 \end{equation}
 \begin{equation}
 \begin{aligned}
\,& dh\otimes dh= \frac{\beta^2}{(1+r^2)^2}dr^2\otimes dr^2, \,& |D h|^2= \frac{4\beta^2 r^2}{(1+r^2)^2}, \nonumber
 \end{aligned}
 \end{equation}
  \begin{equation}
 \begin{aligned}
  D^2 h =\frac{2\beta}{1+r^2}g_0-\frac{\beta}{(1+r^2)^2}dr^2\otimes dr^2, \nonumber
 \end{aligned}
 \end{equation}
 \begin{equation}
 \begin{aligned}
\Delta h=\frac{2n\beta}{1+r^2}-\frac{4\beta r^2}{(1+r^2)^2}=\frac{2\beta (n-2) r^2+2\beta n}{(1+r^2)^2}.  \nonumber
 \end{aligned}
 \end{equation}
 Therefore, we obtain
  \begin{equation}
 \begin{aligned}
%U[h]  :=
\,&
 \frac{\tau-1}{n-2}\Delta h g_0- D^2 h
  +\frac{\tau-2}{2}|Dh|^2 g_0
  + dh\otimes dh
  \\
  =\,& \frac{2\beta}{(1+r^2)^2} \left((\tau-2)(1+\beta)r^2+\frac{n(\tau-2)+2}{n-2} \right)g_0
  +\frac{\beta(1+\beta)}{(1+r^2)^2}dr^2\otimes dr^2, \nonumber
 \end{aligned}
 \end{equation}
 with the eigenvalues (with respect to $g_0$):
  \begin{equation}
 \begin{aligned}
 \lambda_1=\,& \cdots=\lambda_{n-1}=\frac{2\beta}{(1+r^2)^2} \left((\tau-2)(1+\beta)r^2+\frac{n(\tau-2)+2}{n-2} \right),
 \\
 \lambda_n=\,& \frac{2\beta}{(1+r^2)^2} \left( \tau (1+\beta)r^2+\frac{n(\tau-2)+2}{n-2} \right).
 \end{aligned}
 \end{equation}
 
 \subsection{Modified Schouten tensors}
 With replacing \eqref{tau-alpha} by a stronger condition
 \begin{equation}
\label{tau-alpha-2}
 \begin{aligned}
\,& \tau<0, \,&\mbox{ if } \alpha=-1; \\
\,& \tau>\max\left\{2,1+(n-2)(1-\kappa_\Gamma\vartheta_{\Gamma})\right\}, \,&\mbox{ if } \alpha=1,
 \end{aligned}
 \end{equation}
 we prove
 \begin{theorem}
 Suppose \eqref{elliptic}, \eqref{concave}, \eqref{homogeneous-1}, \eqref{tau-alpha-2} hold. Let $\delta>0$ be fixed.
  For any smooth function 
 $\psi$ with 
 \begin{equation}
 \label{delta1}
 \begin{aligned}
 0<\psi(x) \leq \Lambda_2 |x|^{-2-\delta} \mbox{ in } \mathbb{R}^n \mbox{ whenever } |x|\gg1.
 \end{aligned}
 \end{equation}
 There is a smooth admissible complete metric $\tilde{g}$ that is conformal to $g_0$ such that
  \begin{equation}
 \begin{aligned}
 f(\lambda(\tilde{g}^{-1}A_{\tilde{g}}^{\tau,\alpha}))=\psi, \mbox{  } \lambda(\tilde{g}^{-1}A_{\tilde{g}}^{\tau,\alpha})\in\Gamma \mbox{ in } \mathbb{R}^n.  \nonumber
 \end{aligned}
 \end{equation}
 \end{theorem}
 \begin{proof}
  Given \eqref{tau-alpha-2} we can check that
   $\lambda(g_0^{-1}A_{\tilde{g}}^{\tau,\alpha})\geq c_0r^{-2}g_0$ for some $c_0>0$.
   The proof is complete by using Theorem \ref{thm1}.
 \end{proof}
 
 \subsection{Ricci tensor}
  In this  case, $\tau=0, \alpha=-1$. We divide this case into two subcases.
  \begin{theorem}
  Let  $f$ satisfy \eqref{elliptic}, \eqref{concave}, \eqref{homogeneous-1}, and
   we assume that for some $\delta>0$, $\psi$ satisfies \eqref{delta1}. Then there exists a
   smooth admissible complete conformal metric $\tilde{g}=e^{2u}g_0$ satisfying with 
    $$ f(\lambda(-\tilde{g}^{-1}Ric_{\tilde{g}}))=\psi, \mbox{  }
     \lambda(-\tilde{g}^{-1}Ric_{\tilde{g}})\in\Gamma \mbox{ in } \mathbb{R}^n,$$  
    provided that $\Gamma\neq\Gamma_n$.
  \end{theorem}
  \begin{proof}
  The eigenvalues of $-Ric_{\tilde{g}}$ with respect to $g_0$ are as follows:
   \begin{equation}
 \begin{aligned}
 \lambda_1=\,& \cdots=\lambda_{n-1}=\frac{2\beta}{(1+r^2)^2} \left(2(1+\beta)r^2+\frac{2n-2}{n-2} \right)\geq \frac{4\beta(1+\beta)r^2}{(1+r^2)^2},
 \\
 \lambda_n=\,& \frac{4(n-1)\beta}{(1+r^2)^2(n-2)}. \nonumber
 \end{aligned}
 \end{equation}
 Since $\Gamma\neq \Gamma_n$, $(1,\cdots,1,0)\in \Gamma$ and then $c_0':=f(1,\cdots,1,0)>0$ is well defined. So
  \begin{equation}
 \begin{aligned}
 f(\lambda(-g_0^{-1}Ric_{\tilde{g}}))\geq\frac{4 \beta(1+\beta) c_0' r^2}{(1+r^2)^2},\nonumber
 \end{aligned}
 \end{equation}
  that is
  \begin{equation}
 \begin{aligned}
 f(\lambda(-\tilde{g}^{-1}Ric_{\tilde{g}}))\geq\frac{4 \beta(1+\beta) c_0' r^2}{(1+r^2)^{2+2\beta}}. \nonumber
 \end{aligned}
 \end{equation}
 This completes the proof by setting $\beta=\frac{\delta}{4}$.
 \end{proof}

 However, it is much more complicated for the case $\Gamma=\Gamma_n$. While for special cases $f=(\sigma_{n,k})^{1/(n-k)}$,  $0\leq k<n$, 
where $\sigma_{n,k}=C_n^k\sigma_n/\sigma_k$, we obtain the following theorem.
 \begin{theorem}
 There exists a smooth complete metric $\tilde{g}$, which is conformal to $g_0$, with negative Ricci curvature and
  \begin{equation}
 \begin{aligned}
\,& \sigma_{n,k}(-\tilde{g}^{-1}Ric_{\tilde{g}})=\psi \mbox{ in } \mathbb{R}^n, \,& 0\leq k<n, \nonumber
 \end{aligned}
 \end{equation}
 provided  $\psi\in C^\infty(\mathbb{R}^n)$, $0<\psi(x) \leq \Lambda_2 |x|^{-2(n-1-k)-\delta} \mbox{ for } |x|\gg1, \mbox{ for some } \delta>0$.
  \end{theorem}

  When $f$ is the $(k,l)$-quotient functions, $f=(C_n^l\sigma_k/C_n^k\sigma_l)^{1/(k-l)}$, 
 %$0\leq l<k<n$, 
 the problem for negative Ricci tensor on Euclidean spaces was discussed in \cite{Sui2017JGA} by constructing radial symmetric subsolutions and supersolutions.
%There are two important  of complete 
 
\section{Conformal deformation of the $-A_g$}
\label{section6}
 
% When the background space is a compact Riemannian manifold with smooth boundary, the author studied a fully nonlinear version of Theorem \ref{thm1-AM} in \cite{yuan2020conformal} which studies a problem of finding complete conformal metric with prescribed curvature functions %generated by smooth symmetric functions, say $f$, 
%of eigenvalues of Einstein tensor and more general modified Schouten tensors satisfying 
%$\tau>1+(n-2)(1-\kappa_\Gamma\vartheta_\Gamma)$, $\alpha=1$;
%while for negative Ricci tensor the conformal deformation was studied in \cite{Guan2008IMRN,Gursky-Streets-Warren2011}.
%For the $-A_g$, %$A_g=\frac{1}{n-2}(Ric_g-\frac{R_g}{2(n-1)} g)$,
% where $Ric_g$ and $R_g$ denote the Ricci and scalar curvature of $g$,
  The conformally prescribed curvature problem for the $-A_g$ is in general rather hard to handle, 
  since the counterexample of interior estimates %settled by Sheng-Trudinger-Wang
  of solutions to some prescribed curvature equations for conformal deformation of $-A_g$ \cite{ShengTrudingerWang2007}. 
  
    Recall the definition of type 2 cone given by Caffarelli-Nirenberg-Spruck \cite{CNS3}.
\begin{definition}[\cite{CNS3}]\label{def-type2}
$\Gamma$ is said to be of type 1 if the positive $\lambda_i$ axes belong to $\partial\Gamma$; otherwise it is called of type 2.
\end{definition}

  Assuming that $\Gamma$ is of type 2, we study the conformal deformation of  $-A_g$, 
 \begin{equation}
 \label{equ-6}
 \begin{aligned}
\,& f(\lambda(-\tilde{g}^{-1}A_{\tilde{g}}))=\psi, \,& \lambda(g_{\tilde{g}}^{-1}A_{\tilde{g}})\in \Gamma, \mbox{  }\tilde{g}=e^{2u}g \mbox{ in } M. 
 \end{aligned}
 \end{equation}
Under the conformal change $\tilde{g}=e^{2u}g$, by \eqref{conformal-formula1},
% \begin{equation} \begin{aligned}
%- A_{\tilde{g}}  =- A_{g} + \nabla^2 u+\frac{1}{2}|\nabla u|^2 g - du\otimes du. \nonumber
% \end{aligned} \end{equation}
 the equation \eqref{equ-6} is  reduced to 
  \begin{equation}
 \begin{aligned}
f\left(\lambda\left(g^{-1}
 ( \nabla^2 u
  +\frac{1}{2}|\nabla u|^2 g
 - du\otimes du- A_{g})\right)\right)=\psi e^{2u}.  
 \end{aligned}
 \end{equation}

 From Definitions \ref{def-type2} and \ref{yuan-kappa},
  $\kappa_{\Gamma}=n-1$ is equivalent to  $\Gamma$ is of type 2.
  According to Lemma \ref{yuan-k+1}, $f$ is fully uniform ellipticity,
 provided  \eqref{elliptic}, \eqref{concave}, \eqref{addistruc} hold, that is 
 \begin{equation}
 \begin{aligned}
\,& f_i(\lambda)\geq \vartheta_\Gamma \sum_{j=1}^n f_j(\lambda), \,& \forall 1\leq i\leq n,  \nonumber
 \end{aligned}
 \end{equation}
which then allows us to derive the local estimates for gradient and Hessian of solutions.

For the case when $(M,g)$ is a compact Riemannian manifold 
with boundary, as in Section \ref{section3},
% in the present paper and Section 6 of \cite{yuan2020conformal} as well,  
the locally lower and upper barriers are respectively given by $h^+=\log \frac{k\delta^2}{k\rho+\delta^2}$ and $h^-=\beta'\log(1+\frac{\rho}{\delta^2})$ $(0<\beta'\ll1)$. And then we can prove
\begin{theorem}
%\label{thm0-general}
Let $(M,g)$ be a compact Riemannian manifold with smooth boundary $\partial M$
 and $\lambda(-g^{-1}A_g)\in \Gamma$ in $M$.
Suppose \eqref{elliptic}, \eqref{concave}, \eqref{homogeneous-1} hold and that $\psi\in C^\infty(\bar M)$, $\psi>0$ in $\bar M$.
%$\bar M:=M\cup\partial M$.
If, in addition, the corresponding cone $\Gamma$ is of type 2, there is on $M$ a smooth complete conformal metric 
 ${g}_\infty=e^{2u_\infty}g$ %that are conformal to $g$, 
 satisfying 
 \eqref{equ-6}. %and  $\lambda({g}_\infty^{-1}A_{{g}_\infty})\in \Gamma$ in $M$.
 
% In addition, if $\psi|_{\partial M}\equiv 1$ then there is a unique smooth complete conformal metric $g_\infty=e^{2u_\infty}g$ satisfying  \eqref{equ-6} and
% \begin{equation}\label{asymptotic-rate0-general}\begin{aligned}
%\lim_{x\rightarrow\partial M} (u_\infty(x)+\log {\rho}(x))=-\frac{1}{2}\log 2. \nonumber
%\end{aligned}\end{equation} 
%As we denote above,  $\rho(x)$ is the distance function to boundary with respect to $g$.
\end{theorem}
 
 %Let $h=h(\rho)$. Near the boundary, $\Omega_\delta$, $\rho$ is smooth and  
% \begin{equation} \begin{aligned}
% \nabla^2 h+\frac{1}{2}|\nabla h|^2 g-dh\otimes dh=\frac{1}{2}h'^2 |\nabla\rho|^2g+h' \nabla^2 +(h''-h'^2)d\rho\otimes d\rho. \nonumber
%\end{aligned}\end{equation}

%When the background space is complete and noncompact, we obtain
\begin{theorem}
Suppose  \eqref{elliptic}, \eqref{concave}, \eqref{homogeneous-1} hold, and the corresponding cone $\Gamma$ is of type 2. 
Assume that $(M,g)$ is a complete noncompact Riemannian manifold with
 \begin{equation}
 \begin{aligned}
\,& f( \lambda(-g^{-1}A_g))\geq \delta\psi, \,& \lambda(-g^{-1}A_g)\in \Gamma \mbox{ holds for some $\delta>0$,} \mbox{ in } M.
\nonumber
 \end{aligned} 
 \end{equation}
Then there exists at least one smooth function $u_\infty\in C^\infty(M)$ such that $g_\infty=e^{2u_{\infty}}g$ is a smooth admissible complete
 metric satisfying \eqref{equ-6}.
% \begin{equation} \begin{aligned} \,& f( \lambda(-{g}_{\infty}^{-1}A_{g_\infty}))=\psi, \,&  \lambda(-{g}_{\infty}^{-1}A_{g_\infty})\in \Gamma  \mbox{ in } M. \nonumber \end{aligned}  \end{equation}
\end{theorem}

The results %will give a different approach to conformal deformation of 
are applicable to conformal deformation of  Einstein tensor and of modified Schouten tensors $A_g^{\tau,\alpha}$ for  
  $\tau>n-1,$ $\alpha=1$, and $ \tau<1$, $\alpha=-1.$
  % \begin{equation} \begin{aligned}
 %\tau>n-1, \alpha=1,
 %\tau<1, \alpha=-1
 % \end{aligned}
% \end{equation}
 
 \subsubsection{Application to Ricci tensor and certain modified Schouten tensors}
 Let $\varrho$ be a fixed constant with $\varrho<1$ and $\varrho\neq0$.
 Set  $\mu_i=\frac{1}{n-\varrho}(\sum_{j=1}^n\lambda_{j}-\varrho \lambda_i).$ Let $\tilde{f}(\lambda_1,\cdots,\lambda_n)=f(\mu_1,\cdots,\mu_n)$. If $f$ is a homogeneous function of degree one then  so is $\tilde{f}$
 % is a homogeneous function of degree one
 in the corresponding cone 
$\widetilde{\Gamma}=\left\{\lambda:  \mu=(\mu_1,\cdots,\mu_n)\in\Gamma
% \mbox{ for }   \mu_i=\frac{1}{n-\varrho}(\sum_{j=1}^n \lambda_{j} -\varrho\lambda_i) 
\right\}.$
We can check that %the cone $\tilde{\Gamma}$ is of type 2 
$\kappa_{\widetilde{\Gamma}}=n-1$ for any $\Gamma$.
  Again, a straightforward computation gives 
   \begin{equation}
 \begin{aligned}
 \mathrm{tr}(-g^{-1}A_g)g-\varrho (-A_g)=\frac{\varrho}{n-2}Ric_g-\frac{(n-2+\varrho)R_g}{2(n-1)(n-2)}g.  \nonumber
 \end{aligned}
 \end{equation}
 
 As consequences, we derive many results. 
 If $\varrho<0$ then we get the results on $A_g^{\tau,\alpha}$ for $\tau<1, \alpha=-1$; in particular for the special case of 
 $\varrho=-(n-2)$, we obtain the main results proved in \cite{Guan2008IMRN,Gursky-Streets-Warren2011} on a compact Riemannian manifold with boundary.
 %also Theorems \ref{thm2-ricci} and \ref{thm4} above on a complete noncompact Riemannian manifold.  
 When $0<\varrho<1$ we can obtain the main results in \cite{Li2011Sheng} for modified Schouten tensors $A_g^{\tau,\alpha}$ with $\alpha=1, \tau>n-1$ on compact manifolds with boundary.

  \subsubsection{Application to the Einstein tensor}
Let  $\mu_i=\frac{1}{n-1}\sum_{j\neq i}\lambda_{j}$.
Given a homogeneous function of degree one, we let $\tilde{f}(\lambda_1,\cdots,\lambda_n)=f(\mu_1,\cdots,\mu_n)$. Then
$\tilde{f}$ is a homogeneous function of degree one with %the corresponding cone 
$\widetilde{\Gamma}=\left\{\lambda:  \mu=(\mu_1,\cdots,\mu_n)\in\Gamma
% \mbox{ for }   \mu_i=\frac{1}{n-1}\sum_{j\neq i}\lambda_{j} 
\right\}.$
 We can check that when $\Gamma\neq\Gamma_n$,
  $\kappa_{\widetilde{\Gamma}}=n-1$, %and the cone $\widetilde{\Gamma}$ is of type 2, 
  and so $f$ is of fully uniform ellipticity according to Lemma \ref{yuan-k+1}.
  By a simple computation, one derives
  \begin{equation}
 \begin{aligned}
 \mathrm{tr}(-g^{-1}A_g)g-(-A_g)=\frac{1}{n-2}(Ric_g-\frac{R_g}{2}g)=\frac{1}{n-2}G_g. \nonumber
 \end{aligned}
 \end{equation}
 As an application, we obtain %Theorems \ref{thm3} and \ref{thm-n-1} for a complete noncompact Riemannian manifold, and 
   Theorem 1.1 of \cite{yuan2020conformal} on compact manifolds with boundary.

\medskip

%\noindent

\vspace{4mm}

\bigskip

%{\bf Acknowledgement.}

\small
\bibliographystyle{plain}

\end{CJK*}

\end{document}